\documentclass[11pt]{article}
\usepackage{amsmath, amssymb, amsthm, verbatim, enumerate, bbm, color} %DIF >
\usepackage{indentfirst}
\usepackage[vmargin=1in,hmargin=1in]{geometry}
\usepackage{enumitem}
\usepackage{graphicx}
\usepackage{enumitem}

\title{Removal Lemmas with Polynomial Bounds}

\author{Lior Gishboliner \thanks{School of Mathematics, Tel Aviv University, Tel Aviv 69978, Israel.}
\and Asaf Shapira \thanks{
School of Mathematics, Tel Aviv University, Tel Aviv 69978, Israel.
Email: asafico$@$tau.ac.il. Supported in part by ISF Grant 1028/16 and ERC Starting Grant 633509.}
}

\date{\today}
\allowdisplaybreaks
\expandafter\let\expandafter\oldproof\csname\string\proof\endcsname
\let\oldendproof\endproof
\renewenvironment{proof}[1][\proofname]{%
	\oldproof[\bf #1]%
}{\oldendproof}

\newcommand{\ignore}[1]{}

\theoremstyle{plain}
\newtheorem{theorem}{Theorem}
\newtheorem{lemma}{Lemma}[section]
\newtheorem{claim}[lemma]{Claim}
\newtheorem{proposition}[lemma]{Proposition}
\newtheorem{observation}[lemma]{Observation}

\newtheorem{definition}[lemma]{Definition}

\DeclareMathOperator{\poly}{poly}

\newcommand{\A}{\mathcal A}

\newcommand{\C}{\mathcal C}

\newcommand{\F}{\mathcal F}

\newcommand{\N}{\mathcal N}
\newcommand{\Q}{\mathcal Q}
\newcommand{\R}{\mathcal R}
\newcommand{\U}{\mathcal U}

\newcommand{\homleq}{\leq_{\hom}}
\RequirePackage[normalem]{ulem} %DIF PREAMBLE
\RequirePackage{color}\definecolor{RED}{rgb}{1,0,0}\definecolor{BLUE}{rgb}{0,0,1} %DIF PREAMBLE
\begin{document}

\date{}

\maketitle

\begin{abstract}

A common theme in many extremal problems in graph theory is the relation between local and global properties
of graphs. One of the most celebrated results of this type is the Ruzsa-Szemer\'edi triangle removal lemma, which
states that if a graph is $\varepsilon$-far from being triangle free, then most subsets of vertices
of size $C(\varepsilon)$ are not triangle free. Unfortunately, the best known upper bound on $C(\varepsilon)$ is given by
a tower-type function, and it is known that $C(\varepsilon)$ is not polynomial in $\varepsilon^{-1}$.
The triangle removal lemma has been extended to many other graph properties, and for some
of them the corresponding function $C(\varepsilon)$ is polynomial. This raised the natural question, posed by Goldreich in 2005 and more recently by Alon and Fox, of characterizing the properties for which one can prove removal lemmas with polynomial bounds.

Our main results in this paper are new sufficient and necessary criteria for guaranteeing that
a graph property admits a removal lemma with a polynomial bound.
Although both are simple combinatorial criteria, they
imply almost all prior positive and negative results of this type. Moreover, our new sufficient conditions allow us to obtain polynomially bounded removal lemmas for many properties for which
the previously known bounds were of tower-type. In particular, we show that every {\em semi-algebraic} graph property admits a polynomially bounded removal lemma. This confirms a conjecture of Alon.

\end{abstract}

%\setcounter{page}{0}
%\thispagestyle{empty}
%\newpage

\section{Introduction}

%\subsection{Background and motivation}

The relation between local and global properties of graphs lies at the core of many of the most well studied problems
in extremal graph theory. Perhaps the most natural problem of this type is whether the fact that a graph is ``far'' from satisfying
a property ${\cal P}$ implies that it does not satisfy it locally. All graph properties we will consider in this
paper are hereditary (i.e. closed under removal of vertices). Note that for such properties, an induced subgraph of $G$ that does
not satisfy ${\cal P}$ is a ``witness'' to the fact that $G$ itself does not satisfy ${\cal P}$. Thus, for such properties the problem
can be phrased as follows: can we deduce from the fact that $G$ is far from satisfying some hereditary property ${\cal P}$ that $G$
contains a small subgraph which can witness this fact, and moreover, how many such small witnesses does $G$ contain?

%As is turns out, the answer to this question depends on how
%one defines the notion of ``far''. For example, if we define a graph to be far from satisfying ${\cal P}$ if one
%needs to change a constant fraction of its edges in order to make it satisfy ${\cal P}$, then a graph can be far from satisfying ${\cal P}$ yet satisfy it %locally. For example, a famous result of Erd\H{o}s \cite{ErdosLocal} states that for some $c>0$, there are graphs from which one needs to delete at least a %$c$-fraction of the edges in order to make them $3$-colorable, yet every subset of vertices of size $cn$ spans a $3$-colorable subgraph.

Let us turn the above abstract problem into the concrete one we will study in this paper.
We say that a graph $G$ on $n$ vertices is $\varepsilon$-far from satisfying
a property ${\cal P}$ if one needs to add/delete at least $\varepsilon n^2$ edges in order to turn $G$ into a graph satisfying ${\cal P}$.
The following is the local-vs-global problem we will study in this paper.

\begin{definition}\label{def:test} Suppose ${\cal P}$ is a hereditary property. We say that ${\cal P}$ is {\em testable}
if there is a function $f_{\cal P}(\varepsilon)$ so that if $G$ is $\varepsilon$-far from satisfying ${\cal P}$ then the graph spanned by a random sample of $f_{\cal P}(\varepsilon)$ vertices of $G$ does not satisfy ${\cal P}$ with probability at least $2/3$. We say that ${\cal P}$ is {\em easily testable} if $f_{\cal P}(\varepsilon)$ is polynomial in $\varepsilon^{-1}$. If ${\cal P}$ is not easily testable then it is {\em hard to test}.
%We define the same notions with respect to properties ${\cal P}_{\cal F}$ in the same manner.
\end{definition}

Let us mention two famous results in extremal graph theory which fall into the above framework.
The first is the celebrated triangle removal lemma of Ruzsa and Szemer\'edi \cite{RS},
which is usually stated as saying that if a graph $G$ is $\varepsilon$-far from being triangle free, then $G$ contains at least $n^3/f(\varepsilon)$ triangles. It is easy to see that this statement is equivalent to asserting that the property of being triangle free is testable per Definition \ref{def:test} with a similar bound. The original proof of the triangle removal lemma relied on Szemer\'edi's regularity lemma \cite{Szemeredi78}, which supplied tower-type
upper bounds for $f(\varepsilon)$. Due to its intrinsic interest, as well as its relation to other fundamental combinatorial problems,
a lot of effort was put into improving this tower-type bound. Unfortunately, the best known upper bound, due to Fox \cite{Fox11}, is still a tower-type function. At the other direction, it is known that triangle freeness is not easily testable \cite{RS}, but the corresponding
super polynomial lower bound on $f(\varepsilon)$ is very far from the tower-type upper bound (see e.g. Theorem \ref{thm:one_graph_hard} for a similar bound).

A second classical theorem which falls into the framework of Definition \ref{def:test} is a theorem of R\"odl and Duke \cite{RD}, which states that if $G$ is $\varepsilon$-far from being $3$-colorable then $G$ contains a non $3$-colorable subgraph on $f(\varepsilon)$ vertices. Actually, a close inspection of the proof in \cite{RD} reveals that it in fact shows that $3$-colorability is testable. Just as in the case of the triangle removal lemma
discussed above, the original proof in \cite{RD} relied on the regularity lemma and thus supplied only tower-type upper bounds for
$f(\varepsilon)$. However, the situation of this problem changed dramatically when
Goldreich, Goldwasser and Ron \cite{GGR} obtained a new proof of the R\"odl-Duke theorem, which avoided the use of the regularity lemma
and supplied a polynomial upper bound for $f(\varepsilon)$, thus showing that $3$-colorability is easily testable.
Actually, the authors of \cite{GGR} proved a more general result, showing that every so called ``partition property'' is easily testable.

We now pause for a moment and make two observations regarding Definition \ref{def:test}.
We first observe that showing that a hereditary property ${\cal P}$ is testable per Definition \ref{def:test} is equivalent
to proving a removal lemma for ${\cal P}$, that is, to proving that if $G$ is $\varepsilon$-far from satisfying ${\cal P}$ then $G$ contains at least $n^h/g_{\cal P}(\varepsilon)$ induced copies of some graph $H \not \in {\cal P}$ on $h \leq h_{\cal P}(\varepsilon)$ vertices. As it turns out, it will be more convenient to work with Definition \ref{def:test}, especially when dealing with hereditary properties that cannot be characterized by a finite number of forbidden induced subgraphs. The second observation, is that the notion of testability from Definition \ref{def:test} has an interesting algorithmic implication. Suppose we want to design an algorithm that will distinguish with some constant probability, say $2/3$, between graphs satisfying ${\cal P}$
and graphs that are $\varepsilon$-far from satisfying it.
An immediate corollary of the fact that a property is testable, is that one can solve the above relaxed decision
problem in time that depends only on $\varepsilon$ and not on the size of the input. Indeed, all the algorithm has to do is sample
$f_{\cal P}(\varepsilon)$ vertices and check if the induced subgraph spanned by these vertices satisfies ${\cal P}$. Such an algorithm
is called a {\em property tester}, hence the name we used in Definition \ref{def:test}.
This notion of testing graph properties was introduced by Goldreich, Goldwasser and Ron \cite{GGR}. Following \cite{GGR}, numerous other property testing algorithms were designed in various other combinatorial settings.

Given the fact that some hereditary properties are testable, and in light of the algorithmic applications mentioned in the previous paragraph,
it is natural to ask which hereditary properties are testable. This question was answered by
Alon and Shapira \cite{hereditary} who proved that in fact every hereditary property is testable.
This result was later reproved by Lov\'asz and Szegedy \cite{LS}, and generalized to the setting of hypergraphs by R\"odl and Schacht \cite{RodlSch} and by Austin and Tao \cite{AT}. Unfortunately, since all these proofs relied on some form of Szemer\'edi's regularity lemma \cite{Szemeredi78}, the bounds involved are of tower-type. It was also shown in \cite{monotone} that there are cases where bounds of this type are unavoidable. However, the examples involved rely on ad-hoc constructions of families of forbidden subgraphs.

It is thus natural to ask which hereditary graph properties are easily testable, or at least which ``natural''
hereditary properties are easily testable. In other words, for which properties can we prove a removal lemma while avoiding the use of
the regularity lemma. This problem was raised in 2005 by Goldreich \cite{Goldreich} and recently also by Alon and Fox \cite{AF}.
Our main results in this paper address this problem by giving very simple yet general combinatorial sufficient and necessary conditions
for a hereditary property to be easily testable. In particular, we obtain polynomially bounded removal lemmas for many natural graph properties
for which it was not previously known how to obtain a removal lemma without using the regularity lemma.

\subsection{The case of finitely many forbidden subgraphs}\label{subsec:finite}

From this point on, it will be more natural to think of a hereditary property in terms of its forbidden subgraphs.
Given a family of graphs ${\cal F}$, let ${\cal P}^*_{\cal F}$ be the property of
being induced ${\cal F}$-free, i.e. not containing an induced copy of each
of the graphs of ${\cal F}$.
When ${\cal F}$ consists of a single graph $F$ we will use the notation ${\cal P}^*_{F}$.
Note that the family of properties ${\cal P}^*_{\cal F}$ is precisely the family of hereditary properties.

In this subsection we describe our new results concerning hereditary properties which
can be characterized by forbidding a finite number of induced subgraphs, that is, the properties ${\cal P}^*_{\cal F}$
with $\F$ being a finite set.
We will describe both a sufficient and a necessary condition that
a finite family of graphs needs to satisfy in order to guarantee that ${\cal P}^*_{\cal F}$ is easily testable, starting with the former.

We say that a graph $F$ is {\em co-bipartite} if $V(F)$ can be partitioned into
two cliques, and say that $F$ is a {\em split} graph if $V(F)$ can be partitioned into two sets, one spanning
a clique and the other spanning an independent set. Our main positive result regarding finite families is the following
simple combinatorial condition, guaranteeing that ${\cal P}^*_{\cal F}$ is easily testable.

\begin{theorem}\label{thm:easy_finite}
If $\F$ is a finite family of graphs that contains a bipartite graph, a co-bipartite graph and a split graph then
${\cal P}^*_{\cal F}$ is easily testable.
%there are $\varepsilon_0 = \varepsilon_0(\F)$ and

%$c = c_{\ref{thm:easy_finite}}(\F)$ such that for every
%$\varepsilon < \varepsilon_0$ and for every graph $G$ on
%$n \geq n_0(\varepsilon)$ vertices the following holds. If $G$ is $\varepsilon$-far from being $\F$-free then there is
%$F \in \F$ such that $G$ contains at least
%$\varepsilon^{c} n^{v(F)}$ copies of $F$. In particular, $F$-freeness is easily testable.
\end{theorem}

We now mention some immediate applications of Theorem \ref{thm:easy_finite}, starting with known results
that follow as special cases of Theorem \ref{thm:easy_finite}. Let $P_k$ denote the path on $k$ vertices.
Alon and Shapira \cite{induced} proved that ${\cal P}^*_{P_3}$ is easily testable by relying on the fact that a graph satisfies ${\cal P}^*_{P_3}$
if and only if it is a disjoint union of cliques. Observing that $P_3$ is bipartite, co-bipartite and split, Theorem \ref{thm:easy_finite} gives the same result. In the same paper \cite{induced}, it was shown that
for any $F$ other than $P_2,P_3,P_4,C_4$ and their complements, the property ${\cal P}^*_{F}$ is not easily testable. The two cases that were left
open were ${\cal P}^*_{P_4}$ and ${\cal P}^*_{C_4}$. The case of ${\cal P}^*_{P_4}$ was settled only very recently by Alon and Fox \cite{AF}
who used the structural characterization of induced $P_4$-free graphs in order to show that ${\cal P}^*_{P_4}$ is easily testable. As
in the case of $P_3$, since $P_4$ is bipartite, co-bipartite and split, Theorem \ref{thm:easy_finite} gives the result of Alon and Fox \cite{AF} as a special case. Finally, a famous theorem of Alon \cite{subgraphs} states that the property of being (not necessarily induced) $F$-free is easily testable if and only if $F$ is bipartite. It is easy to see that the `if part' of this theorem follows immediately from Theorem \ref{thm:easy_finite}. Indeed, this follows from the simple observation
that being $F$-free is equivalent to satisfying ${\cal P}^*_{\cal F}$, where ${\cal F}$ consists of all supergraphs of $F$ on $|V(F)|$ vertices.

%\begin{theorem}\label{thm:P3}
%There are $\varepsilon_0$ and $c$ such that the following holds for every
%$\varepsilon < \varepsilon_0$ and $n \geq n_0(\varepsilon)$. If a graph $G$ on $n \geq n_0$ vertices is $\varepsilon$-far from being $P_3$-free then $G$ %contains at least $\varepsilon^c n^4$ copies of $P_3$.
%\end{theorem}

Let us turn to derive some new testability results from Theorem \ref{thm:easy_finite}.
It is well known that the property of being a {\em line graph} is equivalent to ${\cal P}^*_{\F}$, where $\F$
is a family of $9$ graphs, each having at most $6$ vertices (see \cite{Harary}). One of these graphs is $K_{1,3}$, which is both bipartite and split, and another one is a complete graph on $5$ vertices minus a single edge, which is co-bipartite.
Hence, Theorem \ref{thm:easy_finite} implies that the property of being a line graph is easily testable.
Two other graph properties which can be shown to be easily testable via Theorem \ref{thm:easy_finite} are being a {\em threshold graph} and a {\em trivially perfect graph}. Since both properties are equivalent to ${\cal P}^*_{\F}$ for an appropriate finite $\F$, where in both cases $P_4 \in \F$ (see \cite{Golumbic, triv_perfect}), we immediately deduce from Theorem \ref{thm:easy_finite} that both are easily testable.

We now turn to describe our necessary condition for being easily testable. Recall that our sufficient condition from Theorem \ref{thm:easy_finite} asks $\F$ to contain a bipartite graph, a co-bipartite graph and a split graph. The next theorem shows that having at least one bipartite graph and at least one co-bipartite graph is a necessary condition.

\begin{theorem}\label{thm:finite_family}
Let $\F$ be a finite family for which ${\cal P}^*_{\F}$ is easily testable. Then
$\F$ contains a bipartite graph and a co-bipartite graph.
\end{theorem}

As we mentioned above, Alon \cite{subgraphs} proved that being $F$-free
is easily testable if and only if $F$ is bipartite. It is now easy to see that the `only if' part of Alon's result
follows from Theorem \ref{thm:finite_family}.
As we mentioned above, Alon and Shapira \cite{induced} proved that ${\cal P}^*_{F}$ is not easily testable for every
$F$ other than $P_2,P_3,P_4,C_4$ and their complements. Again, this result follows as a special case of Theorem \ref{thm:finite_family}.

Having given both a necessary and a sufficient condition, it is natural to ask if one of them in fact characterizes the finite families $\F$ for which
${\cal P}^*_{\F}$ is easily testable. Unfortunately, none do. It is known that being a split graph is equivalent to ${\cal P}^*_{\F}$ where
$\F=\{C_5,C_4,\overline{C_4}\}$ (see \cite{Golumbic}). While $\F$ does not satisfy the condition of Theorem \ref{thm:easy_finite} (it does not contain a split graph), the property
of being a split graph is easily testable since it is one of the partition properties that were shown to be easily testable in \cite{GGR}. Therefore, the sufficient condition in Theorem
\ref{thm:easy_finite} is not necessary. Showing that the necessary condition of Theorem \ref{thm:finite_family} is not sufficient is a bit harder,
and is stated in the following theorem.

\begin{theorem}\label{thm:hard_bp_co-bp}
There is a bipartite $F_1$ and a co-bipartite $F_2$ such that ${\cal P}^*_{\{F_1,F_2\}}$ is not easily testable.
\end{theorem}

Thus the above theorem also implies that in Theorem \ref{thm:easy_finite} we cannot drop the requirement that $\F$ should
contain a split graph. The fact that we cannot drop the requirement that $\F$ should
contain a bipartite graph follows from \cite{RS} where it was (implicitly) proved that
triangle-freeness is not easily testable. By symmetry, the same holds for the co-bipartite graph.

We conclude our discussion on the case of finite forbidden families with the following theorem, which turns out
to be the key step in the proof of Theorem \ref{thm:finite_family}. We will comment on the importance of this theorem
in Subsection \ref{subsec:nuggets}.

\begin{theorem}\label{thm:one_graph_hard}
For every $h \geq 3$ there are $\varepsilon_0 = \varepsilon_0(h)$ and $c = c(h)$ such that the following holds for every $\varepsilon < \varepsilon_0$ and for every non-bipartite graph $H$ on $h$ vertices. For every $n \geq n_0(\varepsilon)$ there is a graph on $n$ vertices which is $\varepsilon$-far from being induced $H$-free and yet contains at most $\varepsilon^{c\log(1/\varepsilon)}n^h$ (not necessarily induced) copies of $H$.
\end{theorem}

%Observe that Theorem \ref{thm:one_graph_hard} implies both the result of \cite{subgraphs}
%on the hardness of testing ${\cal P}_{F}$ for non-bipartite $F$, and the result of \cite{induced}
%on the hardness of testing ${\cal P}^*_{F}$ for $F$ other than $P_2,P_3,P_4,C_4$ and their complements.
%While Theorem \ref{thm:one_graph_hard} produces a graph that is $\varepsilon$-far from being induced $F$-free yet contains
%few copies (induced or not) of $F$, Alon's result \cite{subgraphs} relied on a construction
%of a graph that is $\varepsilon$-far from being $F$-free yet contains few copies of $F$.
%The construction in \cite{subgraph} was thus weaker since it did not
%produce a graph that is $\varepsilon$-far from being {\em induced} $F$-free. In fact, in most cases that graph was induced $F$-free!
%To get a sense as to why our strengthened version of Alon's construction
%was hard to come up with, let us just mention that while the construction stated in Theorem \ref{thm:one_graph_hard} implies Alon's result,
%it also implies the result of \cite{induced} mentioned above, a result which resolved an open problem raised in \cite{subgraphs}.

\subsection{The case of infinitely many forbidden subgraphs}\label{subsec:infinite}

We now turn to consider properties ${\cal P}^*_{\F}$ when $\F$ is a (possibly) infinite family.
We start by introducing an important feature of a hereditary graph property.

\begin{definition}\label{def:f_blowup}
Let $F$ be a graph with vertex set $V(F) = \{1,\dots,p\}$ and let
$g : V(F) \rightarrow \{0,1\}$. We say that a graph $G$ is a {\em $g$-blowup} of $F$ if $G$ admits a vertex partition $V(G) = P_1 \cup \dots \cup P_p$ with the following properties.
\begin{enumerate}
\item For every $1 \leq i < j \leq p$, if $(i,j) \in E(F)$ then $(P_i,P_j)$ is a complete bipartite graph, and if $(i,j) \notin E(F)$ then $(P_i,P_j)$ is an empty bipartite graph.
\item For every $1 \leq i \leq p$, if $g(i) = 1$ then $P_i$ is a clique and if $g(i) = 0$ then $P_i$ is an independent set.
\end{enumerate}
\end{definition}

\begin{definition}\label{def:blowup_property}
We say that a graph property ${\cal P}$ has the {\em blowup quality} if for every graph $F$ which satisfies ${\cal P}$ there is a function
$g : V(F) \rightarrow \{0,1\}$ such that every $g$-blowup of $F$ satisfies ${\cal P}$.
\end{definition}

Our main result regarding hereditary properties characterized by an infinite family of forbidden subgraphs $\F$ is the following.

\begin{theorem}\label{thm:easy_blowup}
Let $\F$ be a graph family such that
\begin{enumerate}
\item  $\F$ contains a bipartite graph, a co-bipartite graph and a split graph.
\item ${\cal P}^*_{\F}$ has the blowup quality.
\end{enumerate}
Then ${\cal P}^*_{\F}$ is easily testable.
\end{theorem}

We now describe what we consider the most important result of this paper.
Let us recall the definition of {\em semi-algebraic graph properties}.
A semi-algebraic graph property ${\cal P}$ is given by an integer $k \geq 1$, a set of real $2k$-variate polynomials $f_1,\dots,f_t \in \mathbb{R}[x_1,\dots,x_{2k}]$ and a Boolean function
$\Phi : \{\text{true},\text{false}\}^t \rightarrow \{\text{true},\text{false}\}$.
A graph $G$ satisfies the property ${\cal P}$ if one can assign a point $p_v \in \mathbb{R}^k$ to each vertex $v \in V(G)$ in such a way that a pair of vertices $u,v$ are adjacent if and only if
\begin{equation*}
\Phi\Big( f_1(p_u,p_v) \geq 0; \dots ; f_t(p_u,p_v) \geq 0 \Big) = \text{true}.
\end{equation*}
In the expression $f_i(p_u,p_v)$, we substitute $p_u$ into the first $k$ variables of $f_i$ and $p_v$ into the last $k$ variables of $f_i$.
In what follows, we call the points $p_v$ {\em witnesses}\footnote{Note that a graph $G$ might have many sets of points witnessing the fact that it satisfies ${\cal P}$.} to the fact that $G$ satisfies ${\cal P}$.

Some examples of semi-algebraic graph properties are those that correspond to being an intersection graph of certain semi-algebraic sets in $\mathbb{R}^k$. For example, a graph is an {\em interval graph} if one can assign an interval in
$\mathbb{R}$ to each vertex so that $u,v$ are adjacent iff their intervals intersect. Similarly, a graph is a {\em unit disc graph}
if it is the intersection graph of unit discs in $\mathbb{R}^2$. The family of semi-algebraic graph properties has been extensively studied
by many researchers, see e.g. \cite{FPS} and its references. Alon \cite{AlPrivate} conjectured that every semi-algebraic graph property is easily testable.
As we now show, this conjecture can be easily derived from Theorem \ref{thm:easy_blowup}.

\begin{theorem}\label{thm:easy_semi_algebraic}
Every semi-algebraic graph property is easily testable.
\end{theorem}

\begin{proof}{\bf (sketch)}
Fix a semi-algebraic graph property ${\cal P}$. Let $\F$ be the family of all graphs which do not satisfy ${\cal P}$. As ${\cal P}$ is a hereditary property we have ${\cal P} = {\cal P}^*_{\F}$. Thus, we only need to show that $\F$ satisfies conditions 1 and 2 in Theorem \ref{thm:easy_blowup}.
The fact that $\F$ satisfies condition 1 of Theorem \ref{thm:easy_blowup} follows directly from the well known fact that every graph satisfying ${\cal P}$
has a bounded VC-dimension. As to condition 2, assume $F$ satisfies ${\cal P}$, and  $\{p_v:v\in V(F)\}$ are points
witnessing this fact. Then setting $g(v)=1$ if and only if
$
\Phi\Big( f_1(p_v,p_v) \geq 0; \dots ; f_t(p_v,p_v) \geq 0 \Big) = \text{true},
$
it is easy to see that every $g$-blowup of $F$ satisfies ${\cal P}$. Indeed, the points witnessing the fact that a $g$-blowup of $F$
satisfies ${\cal P}$ are obtained by taking
each of the points $p_v$ an appropriate number of times.
\end{proof}

The reader can find a more detailed proof of Theorem \ref{thm:easy_semi_algebraic} in Subsection \ref{subsec:semi}.
Returning to the discussion at the beginning of the paper, observe that an immediate corollary of Theorem \ref{thm:easy_semi_algebraic}
is that for every semi-algebraic graph property ${\cal P}$ there is an absolute constant $c$, so that if $G$ is $\varepsilon$-far from satisfying
${\cal P}$, then $G$ contains a subgraph on $\varepsilon^{-c}$ vertices which does not satisfy ${\cal P}$.

Given Theorem \ref{thm:easy_finite}, it is natural to ask if condition $1$ in Theorem \ref{thm:easy_blowup} already guarantees that a property is
easily testable. Actually, an even better reason for believing such a result is the following: as we (implicitly) show later in the paper,
if a hereditary property ${\cal P}$ satisfies condition $1$ of Theorem \ref{thm:easy_blowup} then it has bounded VC dimension\footnote{What we show
(see Lemma \ref{lem:obstruction}) is that condition $1$ implies that every graph $G$ satisfying ${\cal P}$ has no induced copy of some $k \times k$ bipartite graph. It is easy to see that this implies that such a $G$ has VC dimension at most $2k$.}. Thus, stating that condition $1$ in Theorem \ref{thm:easy_blowup}
is a sufficient condition for being easily testable, is equivalent to the (aesthetically pleasing) statement that every hereditary property
of bounded VC dimension is easily testable. As our final theorem shows, this is regretfully not the case.

\begin{theorem}\label{thm:hard_superpoly_blowup}
There is a family of graphs $\F$ that contains a bipartite graph, a co-bipartite graph and a split graph, for which ${\cal P}^*_{\cal F}$ is not easily testable.
\end{theorem}

\subsection{Some nuggets about the proofs}\label{subsec:nuggets}

We start with some comments regarding the proofs of Theorems \ref{thm:easy_finite} and \ref{thm:easy_blowup}.
One key observation needed for these proofs is that given a bipartite graph $A_1$,
a co-bipartite graph $A_2$, and a split graph $A_3$, there is a bipartite graph $B$ on vertex sets $X,Y$, so that no matter which graphs one
puts on $X$ and on $Y$, one always gets a graph containing an {\em induced} copy of either $A_1$, $A_2$ or $A_3$ (see Lemma \ref{lem:obstruction}).
This means that if $\F$ satisfies the assertion of Theorem \ref{thm:easy_finite} and $G$ satisfies ${\cal P}^*_{\F}$ then $G$ has
no induced copy\footnote{Actually, $G$ has no induced copy of any graph obtained by adding edges to the two partition classes of $B$.} of some bipartite graph $B$. If this is the case, then one can apply a ``conditional regularity lemma'' of Alon, Fischer
and Newman \cite{bipartite} in order to find a highly structured partition of $G$ (even more structured than the one produced by the
regularity lemma \cite{Szemeredi78}) which is of size only $\poly(1/\varepsilon)$. This is in sharp contrast to the general argument
of \cite{hereditary} that relied on Szemer\'edi's regularity lemma \cite{Szemeredi78} which can only produce partitions of size $\text{Tower}(1/\varepsilon)$.
The proof of Theorem \ref{thm:easy_blowup} is more involved, mainly due to having to handle an infinite number of forbidden subgraphs.
What usually considerably complicates proofs of this type is the need to embed multiple vertices into the same cluster of the partition mentioned above.
The difficulty arises from the fact that clusters of the partition are not highly structured (as opposed to the bipartite graphs between them).
However, when dealing with properties satisfying condition $2$ of Theorem \ref{thm:easy_blowup}, it is enough to embed at most one vertex
into each cluster. This feature is what makes it possible to prove Theorem \ref{thm:easy_blowup}.

As we mentioned above, the construction described in Theorem \ref{thm:one_graph_hard} is the key step in the
proof of Theorem \ref{thm:finite_family}. Let us explain why in Theorem \ref{thm:one_graph_hard} we managed to overcome a difficulty that was not resolved
in previous works. Alon's result \cite{subgraphs} that being $F$-free is not easily testable for non-bipartite $F$ relied on a construction
of a graph that is $\varepsilon$-far from being $F$-free yet contains only $\varepsilon^{c\log(1/\varepsilon)}n^{v(F)}$ copies of $F$.
He further asked for which $F$ the property ${\cal P}^*_{F}$ is easily testable. The reason why the construction in \cite{subgraphs}
did not imply that ${\cal P}^*_{F}$ is hard for every non-bipartite $F$ (or a complement of one) was that it did not produce a graph
that is $\varepsilon$-far from being {\em induced} $F$-free. In fact, in most cases the graph was induced $F$-free.
So what we do in Theorem \ref{thm:one_graph_hard} is reprove the result of \cite{subgraphs} in a way that simultaneously resolves
the open problem raised in that paper. To prove Theorem \ref{thm:one_graph_hard} we too use a construction based on Behrend's \cite{Behrend}
example of a large set of integers $S$ without $3$-term arithmetic progressions, but with the following twist. First, we take a set $S$ that
does not contain a (non-trivial) solution to {\em any} convex\footnote{A linear equation is convex if it is of the form $a_1x_1+\ldots+a_kx_k=(a_1+\ldots+a_k)x_{k+1}$ with all $a_i > 0$.} linear equation with small coefficients. Second, we carefully label the vertices/clusters in this
construction in such a way that any copy of $H$ in the construction will necessarily contain a {\em monotone} cycle, i.e. a cycle whose labels increase in value. This property guarantees that such a cycle corresponds to a solution of a convex linear equation with integers from $S$, but we know that $S$ has no such solution.

\subsection{Organization}

The rest of the paper is organized as follows. In Section \ref{sec:easy} we prove Theorems
\ref{thm:easy_finite} and \ref{thm:easy_blowup}. We also give a more detailed proof of Theorem \ref{thm:easy_semi_algebraic}
in Subsection \ref{subsec:semi}. In Section \ref{sec:hard} we prove Theorems \ref{thm:finite_family}, \ref{thm:hard_bp_co-bp},  \ref{thm:one_graph_hard} and \ref{thm:hard_superpoly_blowup}.

\section{Easily Testable Properties}\label{sec:easy}

In this section we prove Theorems
\ref{thm:easy_finite} and \ref{thm:easy_blowup}.
We start with some preliminary definitions. Let $G$ be a graph on $n$ vertices.
For a set $X \subseteq V(G)$, we denote by $G[X]$ the subgraph of $G$ induced by $X$. Define
$  e(X) = \left| \left\{ (x,y) \in E(G) : x,y \in X \right\} \right| $
and $d(X) = e(X)/\binom{|X|}{2}$. The number $d(X)$ is called the {\em density} of $X$. Notice that $d(X) = 1$ if and only if $X$ is a clique and
$d(X) = 0$ if and only if $X$ is an independent set. We say that $X$ is {\em homogeneous} if either $d(X) = 1$ or $d(X) = 0$ (i.e. $X$ is either a clique or an independent set). For
$\delta \in (0, \frac{1}{2})$, we say that $X$ is {\em $\delta$-homogeneous} if either $d(X) \geq 1 - \delta$ or $d(X) \leq \delta$.

For two disjoint sets $X,Y \subseteq V(G)$ define
$e(X,Y) = \left| \left\{(x,y) \in X \times Y : (x,y) \in E(G)\right\} \right|$
and $d(X,Y) = \frac{e(X,Y)}{|X||Y|}$. The number $d(X,Y)$ is called the
{\em density} of the pair $(X,Y)$. Note that $d(x,y) = 1$ (resp. $d(x,y) = 0$) if and only if
the bipartite graph between $X$ and $Y$ is complete (resp. empty). We say that the pair $(X,Y)$ is {\em homogeneous} if either $d(X,Y) = 1$ or $d(X,Y) = 0$. For $\delta \in (0, \frac{1}{2})$, we say that $(X,Y)$ is {\em $\delta$-homogeneous} if either
$d(X,Y) \geq 1 - \delta$ or $d(X,Y) \leq \delta$. We say that the {\em dominant value} of $(X,Y)$ is $1$ if $d(X,Y) \geq \frac{1}{2}$ and is $0$ if $d(X,Y) < \frac{1}{2}$.
In cases where we consider several graphs at the same time, we write $d_G(X)$ and $d_G(X,Y)$ to refer to the density in $G$. We will use the following trivial claim.
\begin{claim}\label{claim:hom_subsets}
Let $\beta,\gamma \in (0,\frac{1}{2})$,
let $X,Y$ be disjoint vertex-sets and let $X' \subseteq X$, $Y' \subseteq Y$ be such that $|X'| \geq (\gamma/\beta)^{1/2}|X|$ and $|Y'| \geq (\gamma/\beta)^{1/2}|Y|$. If $d(X,Y) \geq 1 - \gamma$ (resp. $d(X,Y) \leq \gamma$) then $d(X',Y') \geq 1 - \beta$ (resp. $d(X',Y') \leq \beta$).
\end{claim}

The {\em weight} of a pair of disjoint vertex-sets $(X,Y)$ is defined as $\frac{|X||Y|}{n^2}$.
Let ${\cal Q} = \{V_1,\dots,V_q\}$ be a vertex-partition of $G$, i.e.
$V(G) = V_1 \uplus \dots \uplus V_q$. We say that
${\cal Q}$ is an {\em equipartition} if $\left| |V_i| - |V_j| \right| \leq 1$ for every $1 \leq i,j \leq q$.
We say that ${\cal Q}$ is {\em $\delta$-homogeneous} if the total weight of pairs $(V_i,V_j)$,
$1 \leq i < j \leq q$, which are not $\delta$-homogeneous, is at most $\delta$. Note that if ${\cal Q}$ is a $\delta$-homogeneous equipartition then the number of non-$\delta$-homogeneous pairs $(V_i,V_j)$ is at most $2\delta q^2$. In the other direction, if for an equipartition
$\Q = \{V_1,\dots,V_q\}$ the number of non-$\delta$-homogeneous pairs is at most $\delta q^2$, then $\Q$ is $2\delta$-homogeneous.

Let $H = (S \cup T, E)$ be a bipartite graph. A {\em completion} of $H$ is any graph on $V(H)$ that agrees with $H$ on the edges between $S$ and $T$. In other words, a completion of $H$ is any graph obtained by putting two arbitrary graphs on the sets $S$ and $T$. We say that $H$ is a {\em bipartite obstruction} for a graph family $\F$ if {\em every} completion of $H$ is not induced $\F$-free.
The first ingredient in the proofs of Theorems \ref{thm:easy_finite} and \ref{thm:easy_blowup} is the following lemma.

\begin{lemma}\label{lem:obstruction}
If a graph family $\F$ contains a bipartite graph, a co-bipartite graph and a split graph then $\F$ has a bipartite obstruction.
\end{lemma}

\begin{definition}\label{def:bp_copy}
Let $H = (S \cup T, E)$ be a bipartite graph.
An {\em induced bipartite copy} of $H$ in a graph $G$ is an injection
$\varphi : V(H) \rightarrow V(G)$ such that for every
$s \in S$ and $t \in T$ we have $(s,t) \in E(H)$ if and only if
$\left( \varphi(s),\varphi(t) \right) \in E(G)$.
\end{definition}
Notice the difference between an induced copy of $H$ and an induced {\em bipartite} copy of $H$. In an induced copy of $H$, the two sides $S$ and $T$ are mapped to independent sets. In contrast, in an induced bipartite copy there is no restriction on the edges inside $\varphi(S)$ and inside $\varphi(T)$, as the definition is only concerned with the edges between $S$ and $T$.
The following simple claim states that if a graph contains ``many" induced bipartite copies of $H$ then a ``relatively small" sample contains such a copy with high probability.
\begin{claim}\label{claim:sample}
Let $H = (S \cup T, E)$ be a bipartite graph with $|S| = |T| = k$. Suppose that an $n$-vertex graph $G$ contains at least $\alpha n^{2k}$ induced bipartite copies of $H$. Then with probability at least $\frac{2}{3}$, a sample of $4k/\alpha$ vertices from $G$ contains an induced bipartite copy of $H$.
\end{claim}
\noindent
The following lemma is the main tool used in the proofs of Theorems \ref{thm:easy_finite} and \ref{thm:easy_blowup}. For Theorem \ref{thm:easy_finite} we also need Lemma \ref{lem:almost_hom_pairs}.

\begin{lemma}\label{lem:strong_reg}
For every $k \geq 1$ there is $C = C_{\ref{lem:almost_hom_pairs}}(k)$ such that the following holds for every bipartite graph $H = (S \cup T, E)$ with
$|S| = |T| = k$. For every
$\gamma,\delta \in \nolinebreak (0,\frac{1}{2})$, every graph $G$ on
$n \geq \nolinebreak n_0(k,\delta,\gamma)$ vertices either contains at least $(\gamma\delta)^{C}n^{2k}$ induced bipartite copies of $H$ or satisfies the following: there is an equipartition
$\Q = \{Q_1,...,Q_q\}$ of $G$ with
$\delta^{-1} \leq q \leq \delta^{-C}$ parts and there are subsets
$U_i \subseteq Q_i$ such that
\begin{enumerate}
\item For all but at most $\delta q^2$ of the pairs
$1 \leq i < j \leq q$ the following holds: $(Q_i,Q_j)$ is $\delta$-homogeneous and the dominant value of $(U_i,U_j)$ is the same as that of $(Q_i,Q_j)$.
\item $(U_i,U_j)$ is $\gamma$-homogeneous for every $1 \leq i < j \leq q$.
\item $|U_i| \geq \left( \gamma\delta \right)^{C} n$ for every
$1 \leq i \leq q$.
\end{enumerate}
\end{lemma}

\begin{lemma}\label{lem:almost_hom_pairs}
For every $k \geq 1$ there is $C = C_{\ref{lem:almost_hom_pairs}}(k)$ such that the following holds for every bipartite graph $H = (S \cup T, E)$ with
$|S| = |T| = k$. Let $m \geq 1$ be an integer and let
$\alpha \in (0,\frac{1}{2})$. Then for every graph $G$ on
$n \geq \nolinebreak n_0(k,m,\alpha)$ vertices, either $G$ contains at least $\alpha^{C}2^{-Cm}n^{2k}$ induced bipartite copies of $H$ or there are pairwise disjoint subsets $W_1,\dots,W_m \subseteq V(G)$ with the following properties:
\begin{enumerate}
\item Either $d(W_i,W_j) \geq 1 - \alpha$ for every $1 \leq i < j \leq m$ or $d(W_i,W_j) \leq \alpha$ for every $1 \leq i < j \leq m$.
\item $|W_i| \geq \alpha^C 2^{-Cm} n$ for every $1 \leq i \leq m$.
\end{enumerate}
\end{lemma}

\noindent
The last tool we need in the proofs of Theorems \ref{thm:easy_finite} and \ref{thm:easy_blowup} is the following counting lemma.

\begin{lemma}\label{lem:count}
Let $F$ be a graph with $V(F) = \{1,\dots,r\}$ and let
$\lambda \in (0,1)$. Let $W_1,...,W_r$ be pairwise-disjoint vertex sets in an $n$-vertex graph $G$, each of size at least $\lambda n$, such that for every $1 \leq \nolinebreak i < j \leq r$, if $(i,j) \in E(F)$ then
$d(W_i,W_j) \geq 1 - \frac{1}{2r^2}$ and if $(i,j) \notin E(F)$ then
$d(W_i,W_j) \leq \frac{1}{2r^2}$.
Then with probability at least $\frac{2}{3}$, a sample of
$9r/\lambda$ vertices of $G$ contains an induced copy of $F$.
\end{lemma}

\medskip

We are now ready to prove Theorems \ref{thm:easy_finite} and \ref{thm:easy_blowup}.
The proofs of Lemmas \ref{lem:obstruction}, \ref{lem:strong_reg}, \ref{lem:almost_hom_pairs} and \ref{lem:count} are given
in Subsection \ref{subsec:aux_lemmas}.

\begin{proof}[Proof of Theorem \ref{thm:easy_finite}]
Let $\F$ be a finite graph family which contains a bipartite graph, a co-bipartite graph and a split-graph. We prove that
${\cal P} = {\cal P}^*_{\cal F}$ is testable per Definition \ref{def:test} with
$f_{\cal P}(\varepsilon) = \varepsilon^{-c}$ for some $c = c(\F)$.
We assume that $c$ is large enough where needed. By Lemma \ref{lem:obstruction}, $\F$ has a bipartite obstruction
$H = (S \cup T, E)$. We can assume (by adding additional vertices if needed) that $|S| = |T| =: k$.
Set
\begin{equation*} % \label{eq:params_main_thm}
m = \max_{F \in \F}{v(F)}, \; \; \;
C = \max\{ C_{\ref{lem:strong_reg}}(k), C_{\ref{lem:almost_hom_pairs}}(k)\},
\; \; \;
% \gamma = \left( 1/2s^2 \right)^{2C + 1}.
\gamma = \frac{1}{2m^2} \cdot 2^{-8Cm} \;.
\end{equation*}

\noindent
Let $\varepsilon < \frac{1}{2}$ and set $$\delta = \varepsilon/4.$$

Let $G$ be an $n$-vertex graph which is $\varepsilon$-far from being induced $\F$-free. We will assume that $n$ is large enough where needed.
Assume first that $G$ contains at least
$(\gamma\delta/2m)^{2Ckm}n^{2k}$ induced bipartite copies of $H$. Then by Claim \ref{claim:sample}, a sample of
$4k \cdot (2m/\gamma\delta)^{2Ckm}$ vertices from $G$ contains an induced bipartite copy of $H$ with probability at least $\frac{2}{3}$. Since $H$ is a bipartite obstruction for $\F$, every graph which contains an induced bipartite copy of $H$ is not induced $\F$-free. Since $C,m,k,\gamma$ depend only on $\F$ and $\delta$ depends only on $\varepsilon$, and since $\varepsilon < \frac{1}{2}$, we clearly have
$4k \cdot (2m/\gamma\delta)^{2Ckm} \leq \varepsilon^{-c}$
for a large enough $c = c(\F)$. Thus, we established the required result in this case.
	
Suppose from now on that $G$ contains less than
$(\gamma\delta/2m)^{2Ckm}n^{2k}$ induced bipartite copies of $H$.
We apply Lemma \ref{lem:strong_reg} to $G$ with the parameters $\delta,\gamma$ defined above, to get
an equipartition $\Q = \{Q_1,...,Q_q\}$ and subsets
$U_i \subseteq Q_i$ with the properties stated in the lemma.
In particular, for every $1 \leq i \leq q$ we have
$|U_i| \geq (\gamma\delta)^{C_{\ref{lem:strong_reg}}(k)}n \geq (\gamma\delta)^Cn$ (see the choice of $C$), implying that
$(1/2m^2)^{C}2^{-Cm}|U_i|^{2k} \geq
(\gamma\delta/2m)^{2Ckm}n^{2k}$.
Thus, by our assumption above, $G[U_i]$ contains less than
$(1/2m^2)^{C}2^{-Cm}|U_i|^{2k}$ induced bipartite copies of $H$.
By Lemma \ref{lem:almost_hom_pairs}, applied to $G[U_i]$ with $m$ as above and
$\alpha = \frac{1}{2m^2}$, there are disjoint sets
$W_{i,1},\dots,W_{i,m} \subseteq U_i$ with the properties stated in the lemma. In particular, for every $1 \leq j \leq m$ we have
\begin{equation}\label{eq:Wi_size_bound}
|W_{i,j}| \geq
\left( 1/2m^2 \right)^{C_{\ref{lem:almost_hom_pairs}}(k)}
2^{-C_{\ref{lem:almost_hom_pairs}}(k)m} |U_i| \geq
2^{-4Cm}|U_i| \geq
\big(2^{-4m} \gamma\delta \big)^C n.
\end{equation}

\noindent
Let $G'$ be the graph obtained from $G$ by making the following changes.
\begin{enumerate}
\item For every $1 \leq i \leq q$, if
$d(W_{i,j},W_{i,j'}) \geq 1 - \frac{1}{2m^2}$
(resp. $d(W_{i,j},W_{i,j'}) \leq \frac{1}{2m^2}$) for every
$1 \leq \nolinebreak j < \nolinebreak j' \leq m$, then turn $Q_i$ into a clique (resp. an independent set). By Lemma \ref{lem:almost_hom_pairs}, one of these options holds.
\item For every $1 \leq i < j \leq q$, if
$d(U_i,U_j) \geq 1 - \gamma$ then turn $(Q_i,Q_j)$ into a complete bipartite graph and if
$d(U_i,U_j) \leq \gamma$ then turn $(Q_i,Q_j)$ into an empty bipartite graph. By Lemma \ref{lem:strong_reg}, one of these options holds.
\end{enumerate}

We claim that the number of edge-changes made in items 1-2 is less than $\varepsilon n^2$. To prove this, define
$\N$ to be the set of pairs
$1 \leq i < j \leq q$ for which either (a) $(Q_i,Q_j)$ is not $\delta$-homogeneous, or (b)
the dominant value of $(U_i,U_j)$ is not the same as that of $(Q_i,Q_j)$.
By Lemma \ref{lem:strong_reg} we have
$\left| \N \right| \leq \delta q^2$, implying that
$\sum_{(i,j) \in \N}{|Q_i||Q_j|} \leq 2\delta n^2$. Notice that if
$(i,j) \notin \N$ then the number of edge-changes made in the bipartite graph $(Q_i,Q_j)$ is at most $\delta |Q_i||Q_j|$. Therefore, the overall number of edge-changes made in item 2 is at most
$\sum_{i < j}{\delta |Q_i||Q_j|} + \sum_{(i,j) \in \N}{|Q_i||Q_j|} \leq \delta n^2 + 2\delta n^2 = 3\delta n^2$.
Furthermore, since $q \geq 1/\delta$ (by Lemma \ref{lem:strong_reg}), the number of edge-changes made in item $1$ is at most
$q \binom{\left\lceil n/q \right\rceil}{2} <
\nolinebreak n^2/q \leq \nolinebreak \delta n^2$. In conclusion, the number of edge-changes made when turning $G$ into $G'$ is less than
$4\delta n^2 = \varepsilon n^2$.

Since $G$ is $\varepsilon$-far from being induced $\F$-free, $G'$ must contain an induced copy of some $F \in \nolinebreak \F$. Suppose wlog that $Q_{1},\dots,Q_p$ are the parts of $\Q$ which intersect this copy and let $X_{i}$ be the intersection of this copy with $Q_{i}$. From the definition of $G'$ it follows that the sets $X_1,\dots,X_p$ and the bipartite graphs $(X_i,X_j)$, $1 \leq i < j \leq p$, are homogeneous. By our choice of $m$ we clearly have $r := \nolinebreak v(F) \leq \nolinebreak m$.
We claim that the sets
%$(W_{i,\ell})_{1 \leq i \leq s, 1 \leq \ell \leq |X_i|}$
$W_{i,\ell}$, where $1 \leq i \leq p$ and $1 \leq \ell \leq |X_i|$,
satisfy condition 1 of Lemma \ref{lem:count} with respect to $F$ {\em in the graph $G$}.
First, observe that for every $1 \leq i \leq p$,
if $X_i$ is a clique (resp. an independent set) then $d_{G'}(Q_i) = 1$ (resp. $d_{G'}(Q_i) = 0$), implying that
$d_G(W_{i,a},W_{i,b}) \geq 1 - 1/2m^2 \geq 1 - 2/2r^2$
(resp. $d_G(W_{i,a},W_{i,b}) \leq 1/2m^2 \leq 1/2r^2$) for every
$1 \leq a < b \leq m$. Secondly, let $1 \leq i < j \leq p$ and assume wlog that $(X_i,X_j)$ is a complete bipartite graph (the case that $(X_i,X_j)$ is an empty bipartite graph is symmetrical). Then
$d_{G'}(Q_i,Q_j) = 1$, implying that $d_G(U_i,U_j) \geq 1 - \gamma$. By Claim \ref{claim:hom_subsets}, our choice of $\gamma$ and (\ref{eq:Wi_size_bound}), we get that
$d_G(W_{i,a},W_{j,b}) \geq 1 - 1/2m^2 \geq 1 - 1/2r^2$ for every
$1 \leq a,b \leq m$.

Set $\lambda = (2^{-4m}\gamma\delta)^C$. By Lemma \ref{lem:count} and (\ref{eq:Wi_size_bound}), a sample of
$9r/\lambda \leq (9m \cdot 2^{4m}/\gamma\delta)^{C}$ vertices from $G$ contains an induced copy of $F$ (and hence is not induced $\F$-free) with probability at least $\frac{2}{3}$. It is easy to see that
$(9m \cdot 2^{4m}/\gamma\delta)^{C} \leq \varepsilon^{-c}$ for a large enough
$c = c(\F)$. This completes the proof of the theorem.
\end{proof}

\begin{proof}[Proof of Theorem \ref{thm:easy_blowup}]
	Let $\F$ be a graph family as in the statement of the theorem.
	We prove that ${\cal P} = {\cal P}^*_{\cal F}$ is testable per Definition \ref{def:test} with
	$f_{\cal P}(\varepsilon) = \varepsilon^{-c}$ for some $c = c(\F)$. We assume that $c$ is large enough where needed.
	By Lemma \ref{lem:obstruction}, $\F$ has a bipartite obstruction
	$H = (S \cup T, E)$. We can assume that $|S| = |T| =: k$ (by adding additional vertices, if necessary). Throughout the proof, $C$ is the constant $C_{\ref{lem:strong_reg}}(k)$ from Lemma \ref{lem:strong_reg}.
	Let $\varepsilon < \frac{1}{2}$ and set
	\begin{equation*} % \label{eq:params_main_thm_2}
	\delta = \varepsilon/4, \; \; \;
	\gamma = 0.5\delta^{2C}.
	\end{equation*}
	
	Let $G$ be an $n$-vertex graph which is $\varepsilon$-far from satisfying ${\cal P}^*_{\cal F}$. We will assume that $n$ is large enough where needed.
	Suppose first that $G$ contains at least $(\gamma\delta)^{C}n^{2k}$ induced bipartite copies of $H$. By Claim \ref{claim:sample}, a sample of $4k(\gamma\delta)^{-C}$ vertices from $G$ contains an induced bipartite copy of $H$ (and hence is not induced $\F$-free) with probability at least $\frac{2}{3}$. Since $k,C$ depend only on $\F$ and $\gamma,\delta$ depend only on $\F$ and $\varepsilon$, and since $\varepsilon < \frac{1}{2}$, we clearly have $4k (\gamma\delta)^{-C} \leq \varepsilon^{-c}$
	for a large enough $c = c(\F)$. Thus, in this case the theorem holds.
	
	Suppose, then, that $G$ contains less than $(\gamma\delta)^{C}n^{2k}$ induced bipartite copies of $H$.
	We apply Lemma \ref{lem:strong_reg} to $G$ with the parameters $\delta,\gamma$ defined above, to get an equipartition
	$\mathcal{P} = \{Q_1,...,Q_q\}$ and subsets $U_i \subseteq Q_i$ with the properties stated in the lemma.
	
	Define a graph $F$ on $[q]$ as follows. For $1 \leq i < j \leq q$, if $d(U_i,U_j) \geq 1 - \gamma$ then
	$(i,j) \in E(F)$ and if $d(U_i,U_j) \leq \gamma$ then
	$(i,j) \notin E(F)$. We will show that $F$ does not satisfy
	${\cal P}^*_{\cal F}$. Let us first complete the proof based on this fact. By Lemma \ref{lem:strong_reg} we have
	$v(F) = q \leq \delta^{-C}$ and hence $\gamma \leq 1/2q^2$. Now it is easy to see that the sets $U_1,\dots,U_q$ satisfy condition $1$ of Lemma \ref{lem:count} with respect to $F$. By Lemma \ref{lem:strong_reg} we have $|U_i| \geq (\gamma\delta)^C n$ for every $1 \leq i \leq q$. Thus, by Lemma \ref{lem:count}, applied with $\lambda = (\gamma\delta)^C$, a sample of
	$s := 9q(\delta\gamma)^{-C} \leq 9\delta^{-2C}\gamma^{-C}$ vertices from $G$ contains an induced copy of $F$, and hence does not satifsy
	${\cal P}^*_{\cal F}$, with probability at least $\frac{2}{3}$. It is easy to see that
	$s \leq \varepsilon^{-c}$, provided that $c = c(\F)$ is large enough.
	
	It thus remains to show that $F$ does not satisfy ${\cal P}^*_{\cal F}$. Assume, by contradiction, that $F$ satisfies ${\cal P}^*_{\F}$. Since
	${\cal P}^*_{\F}$ has the blowup quality (recall Definition \ref{def:blowup_property}), there is a function
	$g : V(F) \rightarrow \{0,1\}$ such that every $g$-blowup of $F$ satisfies ${\cal P}^*_{\F}$. Now let $G'$ be the graph obtained from $G$ by making the following changes.
\begin{enumerate}
\item For every $1 \leq i \leq q$, if
$g(i) = 1$ then turn $Q_i$ into a clique and if $g(i) = 0$ then turn $Q_i$ into an independent set.
\item For every $1 \leq i < j \leq q$, if $(i,j) \in E(F)$ then turn $(Q_i,Q_j)$ into a complete bipartite graph and if $(i,j) \notin E(F)$
then turn $(Q_i,Q_j)$ into an empty bipartite graph.
\end{enumerate}
Since $G'$ is a $g$-blowup of $F$ (see Definition \ref{def:f_blowup}), $G'$ satisfies ${\cal P}^*_{\F}$.
We now show that the number of edge-changes made in items 1-2 is less than $\varepsilon n^2$, which will stand in contradiction with the fact that $G$ is $\varepsilon$-far from satisfying ${\cal P}^*_{\cal F}$.

By the definitions of $F$ and $G'$ we have the following: for every
$1 \leq i < j \leq q$, if $(Q_i,Q_j)$ is complete in $G'$ then
$d_G(U_i,U_j) \geq 1 - \gamma$, and if $(Q_i,Q_j)$ is empty in $G'$ then $d_G(U_i,U_j) \leq \gamma$.
Define $\N$ to be the set of pairs
$1 \leq i < j \leq q$ for which either (a) $(Q_i,Q_j)$ is not $\delta$-homogeneous, or (b)
the dominant value of $(U_i,U_j)$ is not the same as that of $(Q_i,Q_j)$.
Note that if $(i,j) \notin \N$ then $(Q_i,Q_j)$ is $\delta$-homogeneous (in $G$) with the same dominant value as $(U_i,U_j)$, implying that the number of edge-changes between $Q_i$ and $Q_j$ is at most $\delta |Q_i||Q_j|$.
By Lemma \ref{lem:strong_reg} we have
$\left| \N \right| \leq \delta q^2$, implying that
$ \sum_{(i,j) \in \N}{|Q_i||Q_j|} \leq 2\delta n^2$.
Thus, the overall number of edge-changes made in item 2 above is at most
$\sum_{i < j}{\delta |Q_i||Q_j|} + \sum_{(i,j) \in \N}{|Q_i||Q_j|} \leq \delta n^2 + 2\delta n^2 = 3\delta n^2$. Finally, since
$q \geq 1/\delta$ (see Lemma \ref{lem:strong_reg}), the number of edge-changes made in item $1$ is at most
$q \binom{\left\lceil n/q \right\rceil}{2} < \nolinebreak n^2/q \leq \nolinebreak \delta n^2$. Thus, the overall number of edge-changes made in items 1-2 is less than $4\delta n^2 = \varepsilon n^2$, as required.
\end{proof}
	
\subsection{Proofs of Auxiliary Lemmas}\label{subsec:aux_lemmas}
In this subsection we prove Lemmas \ref{lem:obstruction},  \ref{lem:strong_reg}, \ref{lem:almost_hom_pairs} and \ref{lem:count}, starting with the latter.

\begin{proof} [Proof of Lemma \ref{lem:count}]
For each $1 \leq i \leq r$, sample a vertex $w_i \in W_i$ uniformly at random. For every $1 \leq i < j \leq r$, the assumption of the lemma implies that with probability at least $1 - 1/2r^2$, $(w_i,w_j) \in E(G)$ whenever $(i,j) \in E(F)$ and $(w_i,w_j) \notin E(G)$ whenever $(i,j) \notin E(F)$. By the union bound over all pairs $1 \leq i < j \leq r$ we get the following: With probability at least
$1 - \binom{r}{2}/2r^2 \geq \frac{3}{4}$, the set $\{w_1,\dots,w_r\}$ spans an induced copy of $F$ in which $w_i$ plays the role of $i$.

Now let $S \in \binom{V(G)}{s}$ be a random subset of size
$|S| = s = 9r/\lambda$ and let $\A$ be the event that
$S \cap W_i \neq \emptyset$ for every $1 \leq i \leq r$. In the previous paragraph we proved that conditioned on $\A$ happening, $G[S]$ contains an induced copy of $F$ with probability at least $\frac{3}{4}$. Hence, in order to finish the proof it is enough to show that
$\mathbb{P}[\A^c] \leq \frac{1}{9}$.
%Indeed, this will imply that the probability that $G[Q]$ does not contain an induced copy of $F$ is at most
%$\frac{1}{9} + \frac{8}{9}\frac{1}{4} = \frac{1}{3}$.
For $1 \leq i \leq r$, the probability that $S \cap W_i = \emptyset$ is
$\binom{n - |W_i|}{s}/\binom{n}{s} \leq (1 - \lambda)^s \leq e^{-\lambda s} \leq \frac{1}{9r}$.
%\begin{eqnarray*}
%\frac{\binom{n - |W_i|}{q}}{\binom{n}{q}} \leq (1 - \lambda)^q \leq e^{-\lambda q} \leq \frac{1}{9r}.
%\end{eqnarray*}
Here we used the assumption $|W_i| \geq \lambda n$ and our choice of $s$. By the union bound over all $1 \leq i \leq r$ we get that
$\mathbb{P}[\A^c] \leq r \cdot \frac{1}{9r} = \frac{1}{9}$, as required.
\end{proof}

We now prove Lemma \ref{lem:obstruction}. In the proof we use the following well-known variant of Ramsey's theorem.
\begin{claim}[\cite{Bollobas}]\label{claim:Ramsey}
Every graph on $4^{k}$ vertices contains a homogeneous subset on $k$ vertices.
\end{claim}
	
	\begin{proof}[Proof of Lemma \ref{lem:obstruction}]
	Let $F_1,F_2,F_3 \in \F$ be such that $F_1$ is bipartite, $F_2$ is co-bipartite and $F_3$ is split. Write $V(F_1) = P_1 \cup Q_1$, $V(F_2) = P_2 \cup Q_2$, $V(F_3) = P_3 \cup Q_3$, where $P_1,Q_1,P_3$ are independent sets and $P_2,Q_2,Q_3$ are cliques. Put
	\begin{equation*} % \label{eq:param_k}
	f = v(F_1) + v(F_2) + 2v(F_3),
	\end{equation*}
	and fix an integer $k$ that is divisible by $f$. We will assume that $k$ is large enough where needed. Let $S$ and $T$ be disjoint vertex-sets of size $k$ each and let $H = (S \cup T, E)$ be a random bipartite graph with sides $S$ and $T$; that is, for each $s \in S, t \in T$, the edge $(s,t)$ is added to $H$ randomly and independently with probability $\frac{1}{2}$. We will show that with positive probability, $H$ is a bipartite obstruction for $\F$, thus proving the lemma.
	Throughout the proof we set
	\begin{equation*} % \label{eq:param_r}
	r = \left\lfloor \frac{k - 4^f}{f} \right\rfloor.
	\end{equation*}
An {\em $(f,r)$-family} is a pair $\Q = (\Q_S,\Q_T)$, where
$\Q_S = \{S_1,...,S_{r}\}$ (resp. $\Q_T = \{T_1,...,T_{r}\}$) is a collection of pairwise-disjoint
subsets of $S$ (resp. $T$) of size $f$; that is, $|S_i| = |T_i| = f$ for every $1 \leq i \leq r$. The number of ways to choose an $(f,r)$-family is exactly
\begin{equation}\label{eq:k_partition_count}
\left( \frac{k!}{r! (f!)^{r}(m-fr)!} \right)^2 \leq k^{2k}.
\end{equation}

We need the following variant of Definition \ref{def:bp_copy}. Let $F$ and $H$ be graphs and let $V(F) = P \cup Q$ and
$V(H) = S \cup T$ be vertex-partitions. A {\em copy} of $F[P,Q]$ in $H[S,T]$ is an injection $\varphi : V(F) \rightarrow V(H)$ such that
$\varphi(P) \subseteq S$, $\varphi(Q) \subseteq T$ and for every
$p \in P$ and $q \in Q$ we have $(p,q) \in E(F)$ if and only if
$(\varphi(p),\varphi(q)) \in E(H)$. Note that there is no restriction on the edges inside $\varphi(P)$ and $\varphi(Q)$.
		
Let $\Q$ be an $(f,r)$-family.
For every
$(i,j) \in \left[ r \right] \times \left[ r \right]$, let $A_{\Q}(i,j)$ be the event that $H[S_i,T_j]$ contains copies of $F_1[P_1,Q_1]$, $F_2[P_2,Q_2]$, $F_3[P_3,Q_3]$ and $F_3[Q_3,P_3]$. The following fact follows immediately from the definition of $A_{\Q}(i,j)$.
\begin{observation}\label{obs:copy}
Let $H'$ be a completion of $H$. Suppose that $S_i$ and $T_j$ are homogeneous sets in $H'$ and that $A_{\Q}(i,j)$ happened. Then $H'$ is not induced $\F$-free.
\end{observation}
\begin{proof}
If $S_{i},T_{j}$ are independent sets then
$H'[S_{i} \cup T_{j}]$ contains an induced copy of $F_1$. If $S_{i},T_{j}$ are cliques then $H'[S_{i} \cup T_{j}]$ contains an induced copy of $F_2$. Finally, if $S_{i}$ is a clique and $T_{j}$ is an independent set or vice versa, then $H'[S_{i} \cup T_{j}]$ contains an induced copy of $F_3$.
\end{proof}

Since $|S_i| = |T_j| = f = v(F_1) + v(F_2) + 2v(F_3)$, there is a bipartite graph with sides $S_i,T_j$ that contains copies of $F_1[P_1,Q_1]$, $F_2[P_2,Q_2]$, $F_3[P_3,Q_3]$ and $F_3[Q_3,P_3]$. This implies that
%$A_{\Q}(i,j)$ has positive probability. We have
$\Pr\left[ A_{\Q}(i,j) \right] \geq \nolinebreak 2^{-f^2}$. Since the events
$\left\{ A_{\Q}(i,j) : i,j \in  \left[ r \right]\right\}$ are independent,
the probability that $A_{\Q}(i,j)$ did not happen for any
$(i,j) \in [r] \times [r]$ is at most
$\big( 1 - 2^{-f^2} \big)^{r^2} \leq e^{-2^{-f^2}r^2} < k^{-2k}$,
%\begin{equation*}
%\left( 1 - 2^{-f^2} \right)^{r^2} \leq e^{-2^{-f^2}r^2} < k^{-2k},
%\end{equation*}
with the rightmost inequality holding if $k$ is large enough. By (\ref{eq:k_partition_count}), there are at most $k^{2k}$ ways to choose an $(f,r)$-family $\Q$. By the union bound over all $(f,r)$-families we get that the following holds with positive probability: for every $(f,r)$-family $\Q$, the event $A_{\Q}(i,j)$ happened for at least one $(i,j) \in [r] \times [r]$. We now show that under this condition, $H$ is a bipartite obstruction for $\F$.

Let $H'$ be a completion of $H$. By repeatedly applying Claim \ref{claim:Ramsey} we extract from $S$ pairwise-disjoint homogeneous $f$-sets $S_1,S_2,\dots,S_r$. This is possible due to our choice of $r$.
% ,see (\ref{eq:param_r}).
Similarly, we extract from $T$ pairwise-disjoint homogeneous $f$-sets $T_1,T_2,\dots,T_r$. Put
$\Q_S = \{S_1,...,S_r\}$ and $\Q_T = \{T_1,...,T_r\}$ and consider the $(f,r)$-family $\Q = (\Q_S,\Q_T)$. By our assumption, there is
$(i,j) \in [r] \times [r]$ for which $A_{\Q}(i,j)$ happened.
Since $S_i$ and $T_j$ are homogeneous in $H'$, Observation \ref{obs:copy} implies that $H'$ is not induced $\F$-free.
This completes the proof of the lemma.
\end{proof}
	
	We now get to the proofs of Lemmas \ref{lem:strong_reg} and \ref{lem:almost_hom_pairs}. These lemmas are proved using a ``conditional regularity lemma'' due to Alon, Fischer and Newman \cite{bipartite}. In order to state this lemma we first need some additional definitions. Let $A$ be an $n \times n$ matrix with $0/1$ entries whose rows and columns are indexed by $1,...,n$. For two sets $R,C \subseteq [n]$, the {\em block} $R \times C$ is the submatrix of $A$ whose rows are the elements of $R$ and whose columns are the elements of $C$. The
	{\em dominant value} of a block is the value, $0$ or $1$, that appears in at least half of the entries. For
	$\delta \in (0,\frac{1}{2})$, we say that a block is {\em $\delta$-homogeneous} if its dominant value appears in at least a $(1 - \delta)$-fraction of the entries. If $A$ is the adjacency matrix of a graph $G$ then for every pair of disjoint sets
	$X,Y \subseteq V(G)$, the pair $(X,Y)$ is $\delta$-homogeneous (in the graph sense) if and only if the block $X \times Y$ is $\delta$-homogeneous (in the matrix sense).
%	The following simple fact will be useful in what follows.
%	\begin{claim}
%	Let $\delta,\delta_1,\delta_2 \in (0,\frac{1}{2})$ for which $\delta_1\delta_2 \geq \delta$. Let $X \times Y$ be a $\delta$-homogeneous block and let $X_1 \times Y_1,\dots,X_t \times Y_t$ be disjoint blocks contained in $X \times Y$.
%	\end{claim}
	The weight of a block $R \times C$ is defined as $\frac{|R||C|}{n^2}$. Let
	$\mathcal{R} = \{R_1,...,R_s\}$ and $\mathcal{C} = \{C_1,...,C_t\}$ be partitions of $[n]$. We say that $(\mathcal{R},\mathcal{C})$ is a {\em $\delta$-homogeneous} partition of $A$ if
%	the total weight of blocks $R_i \times C_j$ which are not $\delta$-homogeneous, is at most $\delta$.
	the total weight of $\delta$-homogeneous blocks $R_i \times C_j$ is at least $1 - \delta$.
	
	Let $B$ be a $0/1$-valued $k \times k$ matrix. A {\em copy} of $B$ in $A$ is a sequence of rows $r_1 < r_2 < \dots < r_k$ and a sequence of columns $c_1 < c_2 < \dots < c_k$ such that $A_{r_i,c_j} = B_{i,j}$ for every
	$1 \leq i,j \leq k$. We are now ready to state the Alon-Fischer-Newman Regularity Lemma.
	\begin{lemma}[{\bf Alon-Fischer-Newman \cite{bipartite}}] \label{lem:AFN}
	Let $k \geq 1$ be an integer and let $\delta \in (0,\frac{1}{2})$. Then for every $0/1$-valued matrix $A$ of size $n \times n$ with $n > \left( k / \delta \right)^{ck}$, either $A$ has a $\delta$-homogeneous partition $(\R,\C)$ with $|\R|,|\C| \leq \left( k / \delta \right)^{c k}$ or for every $0/1$-valued
	$k \times k$ matrix $B$, $A$ contains at least
	$\left( \delta / k \right)^{c k^2} n^{2k}$ copies of $B$. Here $c$ is an absolute constant.
	\end{lemma}
 The next lemma is an application of Lemma \ref{lem:AFN} to adjacency matrices of graphs. It is the main tool used in the proofs of Lemmas \ref{lem:strong_reg} and \ref{lem:almost_hom_pairs}.
\begin{lemma}\label{lem:reg}
For every $k \geq 1$ there is $D = D_{\ref{lem:reg}}(k)$ such that the following holds for every bipartite graph $H = (S \cup T, E)$ with
$|S| = |T| = k$. Let $\delta \in (0,\frac{1}{2})$ and let $P \geq 1$ be an integer. Let $G$ be a graph on $n \geq n_0\left(k,\delta,P\right)$ vertices and let ${\cal P}$ be an equipartition of $V(G)$ with $p$ parts, where $p \leq P$. Then either $G$ contains at least $\delta^{D} n^{2k}$ induced bipartite copies of $H$ or $G$ admits a $\delta$-homogeneous equipartition $\Q$ which refines
${\cal P}$ and satisfies $\delta^{-1} \leq |\Q| \leq p \cdot \delta^{-D}$.
\end{lemma}
As the proof of Lemma \ref{lem:reg} is technical, we leave it to the end of this subsection and first prove Lemmas \ref{lem:strong_reg} and \ref{lem:almost_hom_pairs} using Lemma \ref{lem:reg}.
\begin{proof}[Proof of Lemma \ref{lem:strong_reg}]
Throughout the proof, $D = D_{\ref{lem:reg}}(k)$ is the constant from Lemma \ref{lem:reg}.
We prove the lemma with
\begin{equation*}
C = C_{\ref{lem:strong_reg}}(k) = 23D^2.
\end{equation*}
Let us assume that $G$ contains less than $(\gamma\delta)^{C}n^{2k}$ induced bipartite copies of $H$ and prove that the other alternative in the statement of the lemma holds. We assume that $n$ is large enough where needed. Since $(\delta/6)^{D} \geq \delta^{4D} \geq \delta^{C}$, our assumption implies that $G$ contains less than $(\delta/6)^{D}n^{2k}$ induced bipartite copies of $H$.
By applying Lemma \ref{lem:reg} with approximation parameter $\frac{\delta}{6}$ and ${\cal P} = \{V(G)\}$, we obtain a
$\frac{\delta}{6}$-homogeneous equipartition
$\Q = \{Q_1,...,Q_q\}$ of $G$ with $q$ parts, where
$\delta^{-1} \leq q \leq (6/\delta)^{D} \leq \delta^{-4D} \leq \delta^{-C}$.
Set
\begin{equation*} % \label{eq:param_beta_2}
\beta =
\frac{\gamma\delta}{2q^4}.
\end{equation*}
By $q \leq \delta^{-4D}$ we have
$\beta \geq \frac{1}{2}(\gamma\delta)^{16D + 1} \geq
(\gamma\delta)^{18D}$.
Since $\beta^{D} \geq (\gamma\delta)^{18D^2} \geq (\gamma\delta)^C$, our assumption in the beginning of the proof implies that $G$ contains less than
$\beta^{D}n^{2k}$ induced bipartite copies of $H$.
We apply Lemma \ref{lem:reg} to $G$ again, now with approximation parameter $\beta$ and with ${\cal P} = \Q$, to obtain a
$\beta$-homogeneous equipartition $\mathcal{U}$ that refines $\Q$ and satisfies
$\left| \mathcal{U} \right| \leq
\beta^{-D} \left| \Q \right| = \beta^{-D}q \leq (\gamma\delta)^{-22D^2}$.
%Here we used $q \leq \delta^{-4D}$ and $\beta \geq (\gamma\delta)^{18D}$.

For each $1 \leq i \leq q$ define
$\mathcal{U}_i = \{U \in \mathcal{U} : U \subseteq Q_i\}$. Sample a vertex $u_i \in Q_i$ uniformly at random and let $U_i \in \U_i$ be such that
$u_i \in U_i$.  By our choice of
$C = C_{\ref{lem:strong_reg}}(k)$, for every $1 \leq i \leq q$ we have
$|U_i| \geq \big\lfloor \frac{n}{\left| \U \right|} \big\rfloor
\geq \frac{n}{2\left| \U \right|} \geq
\frac{1}{2} \cdot (\gamma\delta)^{22D^2}n \geq
(\gamma\delta)^{C}n$, as required.

Let $\A_1$ be the event that all pairs $(U_i,U_j)$ are $\beta$-homogeneous. For a specific pair $1 \leq i < j \leq q$, the probability that $(U_i,U_j)$ is not $\beta$-homogeneous is
\begin{equation*}
\sum{\frac{|U||U'|}{|Q_i||Q_j|}} \leq
\sum{\frac{|U||U'|}{\left\lfloor n/q \right\rfloor ^2}} \leq
2\sum{\frac{|U||U'|}{(n/q)^2}}
\end{equation*}
where the sum is over all
non-$\beta$-homogeneous pairs
$(U,U') \in \mathcal{U}_i \times \mathcal{U}_j$. This sum is not larger than $2\beta q^2 \leq 1/q^2$ because $\mathcal{U}$ is $\beta$-homogeneous and by our choice of $\beta$.
By the union bound over all pairs
$1 \leq i < j \leq q$ we get that
$\mathbb{P}\left[ \A_1 \right] \geq \frac{1}{2}$. Notice that as
$\beta \leq \gamma,\delta$, if a pair is $\beta$-homogeneous then it is also $\gamma$-homogeneous and $\delta$-homogeneous.

Let $1 \leq i < j \leq q$ for which the pair $(Q_i,Q_j)$ is $\frac{\delta}{6}$-homogeneous. We say that $(Q_i,Q_j)$ is {\em bad} if either
$d(Q_i,Q_j) \geq 1 - \frac{\delta}{6}$ and $d(U_i,U_j) \leq \delta$, or $d(Q_i,Q_j) \leq \frac{\delta}{6}$ and $d(U_i,U_j) \geq 1 - \delta$. Otherwise $(Q_i,Q_j)$ is {\em good}. Let $Z$ be the number of bad pairs $(Q_i,Q_j)$.
Let $\A_2$ be the event that $Z \leq \frac{2 \delta}{3}q^2$.
Consider any $\frac{\delta}{6}$-homogeneous pair $(Q_i,Q_j)$ and assume wlog that $d(Q_i,Q_j) \geq 1 - \frac{\delta}{6}$. Then the probability that $d(U_i,U_j) \leq \delta$ is at most
$\frac{\delta/6}{1 - \delta} < \frac{\delta}{3}$. Therefore
$\mathbb{E}[Z] < \frac{\delta}{3}q^2$ and
% such that the dominant value of $(W_i,W_j)$ is {\em not} the same as that of $(V_i,V_j)$.
by Markov's Inequality
$\mathbb{P}\left[ Z > \frac{2\delta}{3}p^2 \right] < \frac{1}{2}$. This implies that $\mathbb{P}[\A_2] > \frac{1}{2}$.

Since $\mathbb{P}[A_1 \cap A_2] > 0$, there is a choice of $U_1,\dots,U_q$ for which $\A_1$ and $\A_2$ happened.
Since $\Q$ is $\frac{\delta}{6}$-homogeneous, the number of pairs
$1 \leq i < j \leq q$ for which $(Q_i,Q_j)$ is not $\frac{\delta}{6}$-homogeneous is at most $\frac{\delta}{3}q^2$.
Therefore, by the definition of $\A_2$, all but at most at most
$\frac{\delta}{3}q^2 + \frac{2\delta}{3}q^2 = \delta q^2$ of the pairs $(Q_i,Q_j)$ are $\delta$-homogeneous and good. By the definition of $\A_1$, $(U_i,U_j)$ is $\beta$-homogeneous (and hence also $\gamma$-homogeneous and $\delta$-homogeneous) for every $1 \leq i < j \leq q$. If $(U_i,U_j)$ is $\delta$-homogeneous and $(Q_i,Q_j)$ is good then the dominant value of $(U_i,U_j)$ is the same as that of $(Q_i,Q_j)$.
\end{proof}
\begin{proof}[Proof of Lemma \ref{lem:almost_hom_pairs}]
We prove the lemma with
\begin{equation*}
C = C_{\ref{lem:almost_hom_pairs}}(k) = 7D,
\end{equation*}
where $D$ is the constant $D_{\ref{lem:reg}}(k)$ from Lemma \ref{lem:reg}.
Set
\begin{equation}\label{eq:param_beta}
\beta = \alpha 4^{-m-2}.
\end{equation}

If $G$ contains $\beta^D n^{2k}$ induced bipartite copies of $H$ then the assertion of the lemma holds, since
$\beta^D \geq \alpha^{D}2^{-6mD} > \alpha^C 2^{-Cm}$.
Assume then that $G$ contains less than $\beta^D n^{2k}$ induced bipartite copies of $H$.
By Lemma \ref{lem:reg}, applied with approximation parameter $\beta$ and
${\cal P} = \{V(G)\}$, $G$ admits a $\beta$-homogeneous equipartition
$V(G) = W_1 \cup \dots W_w$ with
$w \leq \beta^{-D}$ parts. Note that for every $1 \leq i \leq w$ we have
$|W_i| \geq \left\lfloor n/w \right\rfloor \geq \beta^{D}n/2 \geq
\alpha^C 2^{-Cm}n$, as required.
	
Define an auxiliary graph $R$ on the set $[w]$ in which $(i,j)$ is an edge if and only if the pair $(W_i,W_j)$ is $\beta$-homogeneous. Since $\{W_1,\dots,W_w\}$ is a $\beta$-homogeneous partition, we have
$$e(R) \geq \binom{w}{2} - 2\beta w^2 \geq
\binom{w}{2} - 4^{-m-1}w^2 >
\left( 1 - \frac{1}{4^m - 1} \right)\frac{w^2}{2}.$$
Here we use the inequality $w \geq \frac{1}{\beta} \geq 4^{m+1}$ (see Lemma \ref{lem:reg}).
	
By Turan's Theorem (see, for example, \cite{Bollobas}), $R$ contains a clique $K$ of size $4^m$; assume wlog that
$K = \{1,...,4^m\}$. For every $1 \leq i < j \leq 4^m$, the pair $(W_i,W_j)$ is $\beta$-homogeneous and hence also $\alpha$-homogeneous. Define a new graph on $K$ as follows: put an edge between $i$ and $j$ if and only if
$d(W_i,W_j) \geq 1 - \alpha$ (the other option being that $d(W_i,W_j) \leq \alpha$). By Ramsey's theorem (see Claim \ref{claim:Ramsey}), this graph contains a homogeneous set of size $m$, which we assume, wlog, to be $\{1,\dots,m\}$. Then either
$d(W_i,W_j) \geq 1 - \alpha$ for every
$1 \leq i < j \leq m$ or $d(W_i,W_j) \leq \alpha$ for every
$1 \leq i < j \leq m$, depending on whether $\{1,...,m\}$ is a clique or an independent set.
\end{proof}
\begin{proof}[Proof of Lemma \ref{lem:reg}]
We prove the lemma with
\begin{equation*}
D = D_{\ref{lem:reg}}(k) = 18ck^3,
\end{equation*}
where $c$ is the absolute constant from Lemma \ref{lem:AFN}.
We assume that $G$ contains less than
$\delta^D n^{2k}$ induced bipartite copies of $H$ and prove that $G$ admits a $\delta$-homogeneous equipartition which refines ${\cal P}$ and has
\begin{equation*} % \label{eq:param_p}
q :=
p \cdot \left\lfloor \delta^{-D} \right\rfloor
%q \cdot \left\lceil \left( k / \delta \right)^{18ck} \right\rceil
\end{equation*}
parts. Clearly $\delta^{-1} \leq q \leq
p \cdot \delta^{-D}$, as required.
Write ${\cal P} = \{V_1,...,V_p\}$. By removing at most
$\frac{q}{p}$ vertices from each set $V_i$, we obtain sets
$V'_i \subseteq V_i$ such that
$\left| V'_1 \right| = \left| V'_2 \right| \dots = \left| V'_p \right|$ and $\left| V'_i \right|$ is divisible by $\frac{q}{p}$. Let $G'$ be the
% graph obtained by removing these vertices, namely the
subgraph of $G$ induced by
$V'_1 \cup \dots \cup V'_{p}$.
Set
$m = |V(G')| \geq n - q$ and assume, wlog, that
$V(G') = [m]$.	
We claim that in order to complete the proof it is enough to find a $\frac{\delta}{2}$-homogeneous equipartition $\Q'$ of $G'$ into $q$ equal parts which refines ${\cal P}' := \{V'_1,\dots,V'_p\}$.
Indeed, given this equipartition $\Q'$, we can obtain a $\delta$-homogeneous equipartition $\Q$ of $G$ which refines ${\cal P}$ by doing the following: for each $i = 1,\dots,p$, we distribute the (at most $\frac{q}{p}$) vertices in $V_i \setminus V'_i$ as equally as possible among the $\frac{q}{p}$ parts of $\Q'$ which are contained in $V'_i$. Notice that if $n$ is large enough (as a function of $q$) then $\Q$ would be $\delta$-homogeneous.
Hence, from now on our goal is to prove that $G'$ admits a $\frac{\delta}{2}$-homogeneous equipartition into $q$ equals parts which refines ${\cal P}'$.

Let $A = A(G')$ be the adjacency matrix of $G'$.
Let $B$ be the {\em bipartite adjacency matrix} of $H$; that is, $B$ is a $k \times k$ matrix, indexed by $S \times T$, in which $B_{x,y} = 1$ if $(x,y) \in E(H)$ and $B_{x,y} = 0$ otherwise. Suppose first that $A$ contains $\left(\delta^2 / 4k \right)^{ck^2} {m}^{2k}$ copies of $B$.
A copy of $B$ which does not intersect the main diagonal of $A$ corresponds to an induced bipartite copy of $H$ in $G'$. There are $O(m^{2k-1})$ copies of $B$ which intersect the main diagonal of $A$. Using
$m \geq n - q$ and $\delta < \frac{1}{2}$ and assuming $n$ to be large enough, we conclude that $G'$ (and hence $G$) contains at least
$$\left(\delta^2 / 4k \right)^{ck^2} {m}^{2k} - O(m^{2k-1}) \geq
\left(\delta^2 / 4k \right)^{2ck^2} {n}^{2k} \geq
% \left(\delta^5 / k \right)^{2ck^2} n^{2k} \geq
\delta^{12ck^3} n^{2k} \geq \delta^{D}n^{2k}$$
induced bipartite copies of $H$, in contradiction to our assumption in the beginning of the proof.

Thus, $A'$ contains less than
$\left(\delta^2 / 4k \right)^{ck^2} {m}^{2k}$ copies of $B$.
By Lemma \ref{lem:AFN}, applied with approximation parameter $\frac{\delta^2}{4}$, $A$ admits a
$\frac{\delta^2}{4}$-homogeneous partition $(\R,\C)$ with
$|\R|,|\C| \leq \left( 4k / \delta^2 \right)^{ck}$. Write
$\mathcal{R} = \{R_1,...,R_s\}$ and
$\mathcal{C} = \{C_1,...,C_t\}$. We will use the following inequality, which follows from our choice of $q$ and
$\delta < \frac{1}{2}$.
%\begin{equation}\label{eq:part_size_bound}
%p \geq \frac{16qst}{\delta^2}
%\end{equation}
%which follows from our choice of $p$ and
%$\delta < \frac{1}{2}$:
\begin{equation}\label{eq:part_size_bound}
8pst\delta^{-2} \leq
8p\left( 4k / \delta^2 \right)^{2ck}\delta^{-2}
\leq
p \delta^{-17ck^2} \leq
q.
\end{equation}
	
Recall that
$\left| V'_1 \right| = \left| V'_2 \right| \dots = \left| V'_p \right| = \frac{m}{p}$ and that $\frac{m}{p}$ is divisible by $\frac{q}{p}$. Therefore, each of the sets $V'_{\ell}$ can be partitioned into equal parts of size $\frac{m}{q}$. For every $\ell = 1,\dots,p$, let $\U'_{\ell}$ be the common refinement of the set $V'_{\ell}$ and the partitions $\mathcal{R}$, $\mathcal{C}$, that is
$ \U'_{\ell} = \{V'_{\ell} \cap R_i \cap C_{j} : 1 \leq i \leq s, 1 \leq j \leq t\}.$
As mentioned above, we assume that $V(G') = [m]$ so that the intersection $V'_{\ell} \cap R_i \cap C_j$ makes sense.
For each $U \in \U'_{\ell}$, partition $U$ arbitrarily into parts of size $\frac{m}{q}$ and an additional part $Z_{\ell,U}$ of size less than $\frac{m}{q}$. Let $Z_{\ell}$ be the union of all additional parts $Z_{\ell,U}$, $U \in \U'_{\ell}$. By (\ref{eq:part_size_bound}) we have
\begin{equation}\label{eq:additional_part_bound}
\left| Z_{\ell} \right| < st \frac{m}{q} \leq
\frac{\delta^2 m}{8p}.
\end{equation}
Next, partition $Z_{\ell}$ arbitrarily into parts of size $\frac{m}{q}$. Let  $\Q'_{\ell}$ be the resulting equipartition of $V'_{\ell}$ and put
$\Q' := \bigcup_{{\ell}=1}^{p}{\Q'_{\ell}}$. Then $\Q'$ is an equipartition of $V(G')$ into $q$ equals parts (each of size $\frac{m}{q}$) which refines ${\cal P}' = \{V'_1,\dots,V'_p\}$.
	
To finish the proof, we show that $\Q'$ is $\frac{\delta}{2}$-homogeneous. To this end, let $\N$ be the set of all pairs $(X,Y) \in \Q' \times \Q'$,
$X \neq Y$, which are not $\frac{\delta}{2}$-homogeneous.
Define $Z := \bigcup_{\ell=1}^{p}{Z_\ell}$ and let $\N_1$ be the set of all $(X,Y) \in \N$ such that either $X$ or $Y$ is contained in $Z$. By (\ref{eq:additional_part_bound}) we have
$\left| Z \right| = \sum_{\ell=1}^{p}{|Z_{\ell}|} \leq
\delta^2 m/8$, implying that
\begin{equation}\label{eq:non_hom_parts_count_1}
\sum_{(X,Y) \in \N_1}{|X||Y|} \leq
2m \cdot \left| Z \right| \leq \frac{\delta^2}{4}m^2 .
\end{equation}

By the definition of $\Q'$, for every $X \in \Q'$ which is not contained in $Z$ there are $(i,j) \in [s] \times [t]$ such that
$X \subseteq R_i \cap C_j$. Therefore, for every
$(X,Y) \in \N \setminus \N_1$ there is $(i,j) \in [s] \times [t]$ such that the block $X \times Y$ is contained in the block $R_i \times C_j$. Let $\N_2$ be the set of pairs $(X,Y) \in \N \setminus \N_1$ for which this block
$R_i \times C_j$ (i.e. the block containing $X \times Y$) is $\frac{\delta^2}{4}$-homogeneous. Set
$\N_3 = \N \setminus (\N_1 \cup \N_2)$.
Since $(\mathcal{R},\mathcal{C})$ is a $\frac{\delta^2}{4}$-homogeneous partition we have
\begin{equation}\label{eq:non_hom_parts_count_2}
\sum_{(X,Y) \in \N_3}{\frac{|X||Y|}{m^2}} \leq
\frac{\delta^2}{4}.
\end{equation}

Let $R_i \times C_j$ be a $\frac{\delta^2}{4}$-homogeneous block and assume wlog that the dominant value of
$R_i \times C_{j}$ is \nolinebreak $1$. Let $\N(i,j)$ be the set of pairs
$(X,Y) \in \N_2$ for which $X \times Y \subseteq R_i \times C_{j}$.
Note that if $(X,Y) \in \N_2$ then the fraction of $0$'s in $X \times Y$ is larger than $\delta$. Since the fraction of $0$'s in $R_i \times C_j$ is at most $\delta^2/4$, Markov's inequality gives
$\sum_{(X,Y) \in \N(i,j)}
{\frac{|X||Y|}{|R_{i}||C_{j}|}} \leq \frac{\delta}{4}$.
Summing over all $\frac{\delta^2}{4}$-homogeneous blocks $R_i \times C_j$ gives
\begin{equation}\label{eq:non_hom_parts_count_3}
\sum_{(X,Y) \in \N_2}{\frac{|X||Y|}{m^2}} \leq
\sum_{i=1}^{s}\sum_{j=1}^{t}{\frac{\delta}{4} \cdot \frac{|R_i||C_j|}{m^2}} =
\frac{\delta}{4}.
\end{equation}
Equations (\ref{eq:non_hom_parts_count_1}), (\ref{eq:non_hom_parts_count_2}) and (\ref{eq:non_hom_parts_count_3}) together imply that the total weight of pairs $(X,Y) \in \N$ is at most
$\frac{\delta}{4} + 2 \cdot \frac{\delta^2}{4} \leq \frac{\delta}{2}$, proving that $\Q'$ is $\frac{\delta}{2}$-homogeneous.
\end{proof}

\subsection{Detailed proof of Theorem \ref{thm:easy_semi_algebraic}}\label{subsec:semi}

Let ${\cal P}$ be a semi-algebraic graph property defined by polynomials $f_1,\dots,f_t \in \mathbb{R}_{2k}[x]$ and a boolean function
$\Phi : \{\text{true},\text{false}\}^t \rightarrow \{\text{true},\text{false}\}$. Let $\F$ be the family of all graphs which do not satisfy ${\cal P}$. As ${\cal P}$ is a hereditary property we have ${\cal P} = {\cal P}^*_{\F}$. Thus, we only need to show that $\F$ satisfies conditions 1-2 in Theorem \ref{thm:easy_blowup}.

We start with condition 1. The {\em VC-dimension} of a binary matrix $A$ is the maximal integer $d \geq 0$ for which there is a $d \times 2^d$ submatrix $B$ of $A$, such that the set of columns of $B$ is the set of all $2^d$ binary vectors of length $d$. The VC-dimension of a graph is defined as the VC-dimension of its adjacency matrix. It is known\footnote{This follows from Warren's theorem on sign patterns of systems of polynomials, see for example \cite{Alon_Ramsey}.} that for every semi-algebraic graph property
${\cal P}$ there is $d = d({\cal P})$\footnote{In fact, $d$ only depends on the dimension $k$ and on the number, $t$, and degrees of the polynomials $f_1,\dots,f_t$ used to define ${\cal P}$. } such that every graph which satisfies ${\cal P}$ has VC-dimension strictly less than $d$. Now let $B$ be a $d \times 2^d$ binary matrix whose columns are all $2^d$ binary vectors of length $d$. Let $H$ be a bipartite graph with sides $X = \{x_1,\dots,x_d\}$ and
$Y = \{y_1,\dots,y_{2^d}\}$ such that $(x_i,y_j) \in E(H)$ if and only if $B_{i,j} = 1$. It is easy to see that no matter which graphs one puts on $X$ and on $Y$ (without changing the edges between $X$ and $Y$), the resulting graph on $X \cup Y$ will not satisfy ${\cal P}$ since its VC-dimension will be at least $d = d({\cal P})$. By putting empty graphs on $X$ and on $Y$ we get a bipartite graph that does not satisfy ${\cal P}$. Similarly, by putting complete graphs on $X$ and on $Y$ (a complete graph on $X$ and an empty graph on $Y$) we get a co-bipartite (split) graph which does not satisfy ${\cal P}$. This shows that $\F$ satisfies condition 1 in Theorem \ref{thm:easy_blowup}.

As for condition 2, let $F$ be a graph on $V(F) = [p]$ which satisfies
${\cal P}$ and let $x_1,\dots,x_p \in \nolinebreak \mathbb{R}^k$ be witnesses to the fact that $F$ satisfies ${\cal P}$. That is, for every $1 \leq i < j \leq p$ we have $(i,j) \in E(F)$ if and only if
$
\Phi\Big( f_1(x_i,x_j) \geq 0; \dots ; f_t(x_i,x_j) \geq 0 \Big) = \text{true}
$.
  We define a function
$g : V(F) \rightarrow \{0,1\}$ as follows: $g(i) = 1$ if
\begin{equation*}
\Phi\Big( f_1(x_i,x_i) \geq 0; \dots ; f_t(x_i,x_i) \geq 0 \Big) = \text{true}
\end{equation*}
and $g(i) = 0$ otherwise. We now show that every $g$-blowup of $F$ satisfies ${\cal P}$. Let $G$ be a $g$-blowup of $F$ with a vertex partition $V(G) = P_1 \cup \dots \cup P_p$ (as in Definition \ref{def:f_blowup}). Then for every $1 \leq i \leq p$, we simply assign the point $x_i$ to every vertex of $P_i$. From the definition of a $g$-blowup and from our choice of $g$, it follows that for every
$1 \leq i,j \leq p$ and for every pair of distinct vertices $v_i \in P_i$, $v_j \in P_j$ we have that $(v_i,v_j) \in E(G)$ if and only if
$
\Phi\Big( f_1(x_i,x_j) \geq 0; \dots ; f_t(x_i,x_j) \geq 0 \Big) = \text{true}
$.
Thus we have shown that
${\cal P}$ has the blowup quality, completing the deduction of Theorem \ref{thm:easy_semi_algebraic} from Theorem \ref{thm:easy_blowup}.

\section{Hard to Test Properties}\label{sec:hard}

This section is organized as follows. In Subsection \ref{subsec:Behrend_RS} we describe a variant of the well-known Ruzsa-Szemer\'edi construction which we use in the proofs of Theorems \ref{thm:finite_family}, \ref{thm:hard_bp_co-bp} and \ref{thm:one_graph_hard}. We then prove Theorem \ref{thm:hard_bp_co-bp} in Subsection \ref{subsec:hard_bp_co-bp}. In Subsection \ref{subsec:hom} we introduce some definitions that are needed in order to handle graph families (and not just individual graphs), leading to the proof of Theorem \ref{thm:finite_family} in Subsection \ref{subsec:hard_family}. The main step in the proof of Theorem \ref{thm:finite_family} is Theorem \ref{thm:aux_hard} (see Subsection \ref{subsec:hard_family}), which also implies Theorem \ref{thm:one_graph_hard}. Finally, in Subsection \ref{subsec:superpoly_blowup} we prove Theorem \ref{thm:hard_superpoly_blowup}.

In some of the proofs we use the following simple claim.
	\begin{claim}\label{claim:disj_collection}
	Let $m,h \geq 1$ be integers. Then there is a collection
	$\mathcal{S} \subseteq [m]^{h}$ of size at least $m^2/h^2$ such that every two $h$-tuples in $\mathcal{S}$ have at most one identical entry.
	\end{claim}
	\begin{proof}
	We construct the collection $\mathcal{S}$ greedily: we start with an empty collection, add an arbitrary $h$-tuple to it, discard all $h$-tuples that coincide in more than one entry with the $h$-tuple we added and repeat. At the beginning we have all $m^h$ of the $h$-tuples in $[m]^h$. At each step we discard at most $\binom{h}{2}m^{h-2}$ tuples.
	Therefore, at the end of the process we have a collection of size at least
	\begin{equation*}
	\frac{m^h}{1 + \binom{h}{2}m^{h-2}} \geq \frac{m^h}{h^2m^{h-2}} = \frac{m^2}{h^2},
	\end{equation*}
	as required.
	\end{proof}	

\subsection{The construction of the graph $R$}\label{subsec:Behrend_RS}
We start with the following lemma, which plays a key role in our constructions.
	\begin{lemma}\label{lem:Behrend}
	For every $k \geq 2$ there is $\alpha = \alpha(k)$ such that for every integer $m$ there is a set
	$S \subseteq [m]$, $|S| \geq \frac{m}{e^{\alpha \sqrt{\log m}}}$, with the following property: Let $2 \leq \ell \leq k$ and let
	$a_1,...,a_{\ell} \geq 1$ be integers satisfying $a_1 + \dots + a_{\ell} \leq k$. Then the only solutions to the equation
	$$a_1s_1 + a_2s_2 + \dots + a_{\ell}s_{\ell} =
	(a_1 + \dots + a_{\ell})s_{\ell+1}$$
	with $s_1,...,s_{\ell + 1} \in S$ are trivial, i.e.
	$s_1 = s_2 = \dots = s_{\ell} = s_{\ell + 1}$.
	\end{lemma}

Lemma \ref{lem:Behrend} is a variant of Behrend's construction \cite{Behrend} of a large subset of $[m]$ without a $3$-term arithmetic progression.
It is easy to show (see e.g. \cite{Ruzsa} and \cite{subgraphs}) that the same exact proof actually works
for any fixed convex equation, and that moreover, it works ``simultaneously'' for all convex equations (for fixed $k$) thus giving the above Lemma.
We thus omit a proof of the lemma.

The following lemma is our variant of the Ruzsa-Szemer\'edi construction, and is the key ingredient in the proofs of Theorems \ref{thm:finite_family}, \ref{thm:hard_bp_co-bp} and \ref{thm:one_graph_hard}.
	\begin{lemma}\label{lem:construction}
	For every $h \geq 3$ there are $\delta_0 = \delta_0(h)$ and
	$\beta = \beta(h)$ such that for every $\delta < \delta_0$ there is a graph $R = R(h,\delta)$ with the following properties:
	\begin{enumerate}
	\item $V(R) = V_1 \cup \dots \cup V_{h}$ and $V_i$ is an independent set for each $1 \leq i \leq h$.
	\item $|V(R)| \geq (1/\delta)^{\beta\log(1/\delta)}$.
	\item $E(R)$ is the union of at least $\delta |V(R)|^2$ pairwise edge-disjoint $h$-cliques.
	\item For every $3 \leq t \leq h$ and for every sequence
	$1 \leq i_1 < i_2 \dots < i_t \leq h$, $R$ contains at most
	$|V(R)|^2$ (not necessarily induced) cycles of the form
	$v_{i_1}v_{i_2} \dots v_{i_{t}}v_{i_1}$ with $v_{i_j} \in V_{i_j}$.
	\end{enumerate}
	\end{lemma}
	\begin{proof}
	Let $\delta > 0$ and let $m$ be the largest integer satisfying
	\begin{equation}\label{eq:m_choice}
	\delta \leq \frac{1}{(h+1)^4e^{\alpha \sqrt{\log m}}}
	\end{equation}
	where $\alpha = \alpha(h-1)$ is from Lemma \ref{lem:Behrend}.
	It is easy to check that
	\begin{equation}\label{eq:bound_m}
	m \geq
	e^{\alpha^{-2}\log^2\left( \frac{1}{\delta(h+1)^4} \right)}
	\geq \left( 1/\delta \right)^{\beta\log(1/\delta)}
	\end{equation}
	if $\beta$ and $\delta$ are sufficiently small.
	
	Let $S \subseteq [m]$ be the set obtained by applying Lemma \ref{lem:Behrend} with $k = h-1$. For each
	$j = 1,\dots,h$ set $V_{j} = \{1,2,...,jm\}$. With a slight abuse of notation, we think of $V_1,...,V_h$ as disjoint sets. Set
	$V(R) = V_1 \cup \dots \cup V_h$. By (\ref{eq:bound_m}) we have
	$|V(R)| = \binom{h+1}{2}m \geq m \geq
		\left( 1/\delta \right)^{\beta\log(1/\delta)}$, as required.
%	\begin{equation*}
%	|V(R)| = \binom{h+1}{2}m \geq m \geq
%	\left( \frac{1}{\delta} \right)^{\beta\log(1/\delta)}.
%	\end{equation*}
	
	For every $x \in [m]$ and $s \in S$ define $A(x,s) = \{x,x+s,x+2s,\dots,x+(h-1)s\}$ and put a clique on $A(x,s)$ in $R$, where
	$x + js$ is taken from $V_{j+1}$ for every
	$j = 0,\dots,h-1$. Notice that for every $(x,s), (x',s') \in [m] \times S$ the following holds: if $\left| A(x,s) \cap A(x',s') \right| \geq 2$ then $(x,s) = (x',s')$. Indeed, if $\left| A(x,s) \cap A(x',s') \right| \geq 2$ then there are $0 \leq i < j \leq h-1$ for which $x + is = x' + is'$ and
	$x + js = x' + js'$. Solving this system of equations yields
	$(x,s) = (x',s')$, as required. So the cliques that we defined are indeed edge-disjoint. By the lower bound on $|S|$ in Lemma \ref{lem:Behrend} and by (\ref{eq:m_choice}), the number of these cliques is
	\begin{equation*}
	m \cdot |S| \geq \frac{m^2}{e^{\alpha \sqrt{\log m}}} \geq
	\delta (h+1)^4 m^2 \geq \delta |V(R)|^2\;.
	\end{equation*}
	
	To finish the proof, we show that for every $t \geq 3$, for every sequence $1 \leq i_1 < i_2 \dots < i_t \leq h$ and for every cycle of the form $v_{i_1}v_{i_2} \dots v_{i_t}v_{i_1}$ with
	$v_{i_j} \in V_{i_j}$, there are $x \in [m]$ and $s \in S$ such that
	$v_{i_1},v_{i_2},\dots,v_{i_t} \in A(x,s)$. This will show that the cycles of this form are pair-disjoint, implying that their number is at most $|V(R)|^2$.
	
	Let $v_{i_1}v_{i_2} \dots v_{i_t}v_{i_1}$ be a cycle in $R$ with
	$v_{i_j} \in V_{i_j}$ for every $1 \leq j \leq t$. By the definition of $R$, for every $j = 1,\dots,t$ there is $(x_j,s_j) \in [m] \times S$ such that $\{v_{i_{j}},v_{i_{j+1}}\} \subseteq A(x_j,s_j)$, with indices taken modulo $t$. This means that for every $1 \leq j \leq t-1$ we have
	\begin{equation}\label{eq:RS_graph_difference}
	\; v_{i_{j+1}} - v_{i_{j}} = (i_{j+1} - i_j) s_j,
	\end{equation}
	and that
	$v_{i_t} - v_{i_1} = (i_t - i_1)s_{t}$. Now, setting
	$a_j = i_{j+1} - i_j$ for $1 \leq j \leq t-1$ we get that
	$$ a_1s_1 + a_2s_2 + \dots + a_{t-1}s_{t-1} =
	(a_1 + \dots + a_{t-1})s_{t}. $$
	By using the property of $S$, stated in Lemma \ref{lem:Behrend}, with
	$\ell = t-1$, we conclude that
	$$s_1 = s_2 = \dots = s_t =: s.$$
	Now (\ref{eq:RS_graph_difference}) implies that
	$v_{i_j} = v_{i_1} + (i_j - i_1)s$ for every $1 \leq j \leq t$. By the definition of $A(x_1,s)$, the fact that $v_{i_1} \in A(x_1,s)$ implies that $v_{i_1},v_{i_2},\dots,v_{i_t} \in A(x_1,s)$, as required.
	\end{proof}
	
\subsection{Proof of Theorem \ref{thm:hard_bp_co-bp}}\label{subsec:hard_bp_co-bp}
	
	Consider the graph with vertices $\{1,2,3,4,5,6,7\}$ and edges $\{1,2\}, \{3,4\}, \{5,6\}$. Let $M$ be the complement of this graph. It is easy to see that $M$ is co-bipartite. We will prove Theorem \ref{thm:hard_bp_co-bp} with
	$F_1 = C_8$ (the cycle on $8$ vertices) and $F_2 = M$. We need the following lemma, which we prove later.
	
	\begin{lemma}\label{lem:copies}
	Let $G$ be a graph with a vertex partition $V = X_1 \cup ... \cup X_8$ such that
	\begin{itemize}
	\item $X_1,X_3,X_5,X_7$ are cliques and $X_2,X_4,X_6,X_8$ are independent sets.
	\item There are only edges between consecutive parts, i.e. $E(X_i,X_j) = \emptyset$ unless $|i-j| \equiv \pm 1 \pmod{8}$.
	\end{itemize}
	Then the following hold.
	\begin{enumerate}
	\item Every induced copy of $C_8$ in $G$ is of the form $x_1x_2...x_8x_1$, where $x_i \in X_i$.
	\item $G$ is induced $M$-free.
	\end{enumerate}
	\end{lemma}
	
	\begin{proof}[Proof of Theorem \ref{thm:hard_bp_co-bp}]
	Set $F_1 = C_8$ and $F_2 = M$.
%Fix $\varepsilon_0 = \varepsilon_0(h)$ so that every $\varepsilon < \varepsilon_0$ will satisfy the inequalities
%\begin{equation}\label{eq:eps_0_2}
%(64\varepsilon)^{\log(1/64\varepsilon)} \leq \varepsilon^{1/2\log(1/\varepsilon)},
%\; \; \;
%\varepsilon < \frac{\delta_0(8)}{64},
%\end{equation}
%where $\delta_0(8)$ is from Lemma \ref{lem:construction}.
%Set $c = \frac{\beta}{8}$,
%where $\beta = \beta(h)$ is from Lemma \ref{lem:construction}.
We will show that for every sufficiently small $\varepsilon > 0$ and for every $n \geq n_0(\varepsilon)$ there is a graph $G$ on $n$ vertices which is $\varepsilon$-far from being induced $\{F_1,F_2\}$-free yet contains at most $c\varepsilon^{\frac{1}{c}\log(1/c\varepsilon)}n^{v(F_i)}$ induced copies of $F_i$
%\footnote{In fact, $G$ will be induced $F_2$-free.}
for $i = 1,2$ and for some $c = c(F_1,F_2)$ \footnote{In fact, $G$ will be induced $F_2$-free.}. This will imply that
${\cal P}^*_{\{F_1,F_2\}}$ is not easily testable.
	
Let $\varepsilon \in \big(0, \frac{\delta_0(8)}{64} \big)$, where $\delta_0(8)$ is from Lemma \ref{lem:construction}. Let
 $R = R(8,64\varepsilon)$ be the graph obtained by applying Lemma \ref{lem:construction}. Recall that
$V(R) = V_1 \cup \dots \cup V_8$ and put $r = |V(R)|$. For simplicity of presentation, we assume that $n$ is divisible by $r$. We define a graph $G$ on an $\frac{n}{r}$-blowup of $R$, that is, we replace each vertex $v \in V(R)$ with a vertex-set $B(v)$ of size $\frac{n}{r}$.
Put $X_i = \bigcup_{v \in V_i}{B(v_i)}$ for $1 \leq i \leq 8$. The edges of $G$ are defined as follows: $B(V_1),B(V_3),B(V_5),B(V_7)$ are cliques and $B(V_2),B(V_4),B(V_6),B(V_8)$ are independent sets. By Lemma \ref{lem:construction}, $R$ contains a collection $\mathcal{H}$ of at least $64\varepsilon r^2$ pairwise edge-disjoint cliques, each of the form $\{v_1,\dots,v_8\}$ with $v_i \in V_i$. For each such clique
$\{v_1,\dots,v_8\} \in \mathcal{H}$ we put a blowup of $C_8$ on the sets $B(v_1),\dots,B(v_8)$, namely, for each
$(x_1,\dots,x_8) \in B(v_i) \times \dots \times B(v_8)$,
$x_1x_2\dots x_8x_1$ is an induced $8$-cycle in $G$. Notice that $G$ satisfies the assumptions of Lemma \ref{lem:copies} with
$X_i = B(V_i)$. Thus, $G$ is induced $M$-free and every induced copy of $C_8$ in $G$ is of the form $x_1x_2 \dots x_8x_1$ with $x_i \in B(V_i)$. Let $x_1x_2 \dots x_8x_1$ be an induced copy of $C_8$ in $G$ and let $v_i \in V_i$ be such that $x_i \in B(v_i)$. From the construction of $G$ it follows that
$v_1v_2\dots v_8v_1$ is a (not necessarily induced) cycle in $R$. By item 4 in Lemma \ref{lem:construction}, the number of such cycles is at most $r^2$. We conclude that $G$ contains at most
$r^2 \left( n/r \right)^8 \leq n^8/r$ induced copies of $C_8$. By item $2$ in Lemma \ref{lem:construction} we have
$r \geq (1/64\varepsilon)^{\beta \log(1/64\varepsilon)}$. Therefore, the number of induced copies of $C_8$ in $G$ is at most
$(64\varepsilon)^{\beta\log(1/64\varepsilon)}n^8$, as required.
%	\begin{equation*}
%	\delta^{\beta \log(1/\delta)}n^8 = (64\varepsilon)^{\beta\log(1/64\varepsilon)}n^8 \leq
%	\varepsilon^{\frac{\beta}{2}\log(1/\varepsilon)}n^8,
%	\end{equation*}
%	provided that $\varepsilon$ is small enough.
	
	To finish the proof, we show that $G$ contains $\varepsilon n^2$  pair-disjoint\footnote{Two subgraphs are {\em pair disjoint} if they share at most one vertex.} induced copies of $C_8$, which will imply that $G$ is $\varepsilon$-far from being induced $\{C_8,M\}$-free.
%	Recall the definition of $\mathcal{H}$ and that
%	$\left| \mathcal{H} \right| \geq 64 \varepsilon ^2$.
	By Claim \ref{claim:disj_collection} and the construction of $G$, for every clique $\{v_1,\dots,v_8\} \in \mathcal{H}$ there are
	$\left( n/8r \right)^2$ pair-disjoint induced copies of $C_8$ of the form $(x_1,\dots,x_8) \in B(v_i) \times \dots \times B(v_8)$. Since the cliques in $\mathcal{H}$ are pair-disjoint, copies of $C_8$ that come from different cliques are pair-disjoint. By using
	$|{\cal H}| \geq 64 \varepsilon r^2$ we get that $G$ contains a collection of
	$|\mathcal{H}| \cdot \left( n/8r \right)^2 \geq
	64 \varepsilon r^2 \left( n/8r \right)^2 = \varepsilon n^2$
%$$ |\mathcal{H}| \cdot \frac{n^2}{64r^2} \geq
%\delta r^2 \cdot \frac{n^2}{64r^2} = \frac{\delta n^2}{64} =
%\varepsilon n^2$$
pair-disjoint induced copies of $C_8$, as required.
	\end{proof}
\begin{proof}[Proof of Lemma \ref{lem:copies}]
We start with item $1$. Let $C = x_1x_2...x_8x_1$ be an induced copy of $C_8$ in $G$. Our goal is to show that
$|C \cap \nolinebreak X_i| = \nolinebreak 1$ for every $1 \leq i \leq 8$. First, assume by contradiction, that
$\left| C \cap X_i \right| \geq 2$ for some $i \in \nolinebreak \{1,3,5,7\}$. Note that $|C \cap X_i| < 3$, as otherwise $C$ would contain a triangle. As $X_i$ is a clique, there is $j \in \{1,...,8\}$ for which $x_j,x_{j+1} \in X_i$ (with indices taken modulo $8$). We may assume, wlog, that $x_1,x_2 \in X_1$. As $x_3,x_8 \notin X_1$, we must have
$x_3,x_8 \in X_2 \cup X_8$. First we consider the case that $x_3$ and $x_8$ are in the same part, say $x_3,x_8 \in X_2$. Then $x_4,x_7 \in X_3$ because
$(x_4,x_3), (x_7,x_8) \in E(G)$, $X_2$ is an independent set and
$x_4,x_7 \notin X_1$. Since $X_3$ is a clique, we get that
$(x_4,x_7) \in E(G)$, in contradiction to the fact that $C$ is an induced cycle. Now we consider the case that $x_3$,$x_8$ are in different parts, say
$x_3 \in X_2$, $x_8 \in X_8$. The path $P = x_3x_4...x_8$ cannot go through $X_1$. Therefore, $P$ must contain at least one vertex from each of the seven parts $X_2,...,X_8$. However, this is impossible since $P$ contains $6$ vertices.

In the previous paragraph we showed that $|C \cap X_i| \leq 1$ for every
$i \in \{1,3,5,7\}$.  Define the sets
$X_{odd} := X_1 \cup X_3 \cup X_5 \cup X_7$ and
$X_{even} := \nolinebreak X_2 \cup X_4 \cup X_6 \cup X_8$. Since
$X_{even}$ is an independent set and $\alpha(C_8) = \nolinebreak 4$, we have $|C \cap X_{even}| \leq 4$. Thus
$|C \cap X_{odd}| \geq 4$, implying that
$|C \cap X_i| = 1$ for every $i \in \nolinebreak \{1,3,5,7\}$.   	
In order to finish the proof it is enough to show that $|C \cap X_i| \geq 1$ for each $i \in \nolinebreak \{2,4,6,8\}$. Suppose, by contradiction, that
$C \cap X_i = \emptyset$ for some $i \in \nolinebreak \{2,4,6,8\}$, say $i = 2$. Let $j,k \in \{1,\dots,8\}$ be such that $x_j \in X_1$ and
$x_k \in X_3$. In the cycle $C$ there is a path between $x_j$ and $x_k$ with at most $5$ vertices (including $x_j$ and $x_k$). This path cannot intersect $X_2$, so it must contain at least one vertex from each of the seven parts $X_1,X_3,X_4,...,X_8$, which is impossible.

For item $2$, suppose by contradiction that $Y \subseteq V(G)$ spans an induced copy of $M$. As before, define $X_{odd} = X_1 \cup X_3 \cup X_5 \cup X_7$ and $X_{even} = X_2 \cup X_4 \cup X_6 \cup X_8$. Notice that $X_{even}$ is an independent set and that $X_{odd}$ is a disjoint union of cliques and hence induced $P_3$-free (where $P_3$ is the path with $3$ vertices). Observe that every set of 5 vertices of $M$ contains an induced copy of $P_3$. So
$|Y \cap X_{odd}| \leq 4$. Moreover, $|Y \cap X_{even}| \leq 2$ because $\alpha(M) = 2$. We got that $|Y| \leq 6 < 7 = |V(M)|$, a contradiction.
\end{proof}	
	
	\subsection{Homomorphisms and Cores} \label{subsec:hom}
	Recall that a {\em homomorphism} from a graph $G_1$ to a graph $G_2$ is a map $f: V(G_1) \rightarrow V(G_2)$ such that for every $u,v \in V(G_1)$, if $(u,v) \in E(G_1)$ then $(f(u),f(v)) \in E(G_2)$.
%	Notice that a bijective homomorphism from a graph to itself is necessarily an isomorphism.
	For a graph $G$, the {\em core} of $G$, denoted $C(G)$, is the smallest induced subgraph of $G$ (with respect to number of vertices) to which there is a homomorphism from $G$.
%	Notice that by definition, there is no homomorphism from $C(G)$ to any proper subgraph of it. In other words, every homomorphism from the core of $G$ to itself is an isomorphism.
%	We say that a graph is a {\em core} if it is the core of itself.
	Observe that every homomorphism from $C(G)$ to itself is an isomorphism.
	We write $G_1 \homleq G_2$ if there is a homomorphism from $G_2$ to $G_1$, and we say that $G_2$ is homomorphic to $G_1$. Notice that the relation $\homleq$ is transitive.
	The notation $G_1 \cong G_2$ means that $G_1$ and $G_2$ are isomorphic graphs.
	
	Let $\F$ be a finite family of graphs and consider the set
	$\C = \C(\F) = \{C(F) : F \in \F \}$. It is easy to see that $(\C,\homleq)$ is a poset in the following sense:
	for every $C_1,C_2 \in \C$, if $C_2 \homleq C_1$ and $C_1 \homleq C_2$ then $C_1 \cong C_2$. Indeed, if there are homomorphisms $f : C_1 \rightarrow C_2$ and $g : C_2 \rightarrow C_1$ then
	$g \circ f$ (resp. $f \circ g$) is a homomorphism from $C_1$ (resp. $C_2$) to itself. Thus $f \circ g$ and $g \circ f$ are isomorphisms, implying that so are $f$ and $g$.
	
	In other words, $\homleq$ is a partial order on the set of equivalence classes of $\C$ under the equivalence relation of graph isomorphism.
	Let $K(\F)$ be a maximal element of the poset $(\C,\homleq)$, i.e. $K(\F)$ is an (arbitrary) element of a maximal equivalence class. The maximality of $K(\F)$ implies that for every $C \in \C$, if there is a homomorphism from $C$ to $K(\F)$ (namely if $K(\F) \homleq C$) then $C \cong K(\F)$.
	\begin{proposition}\label{prop:core}
	Let $F \in \F$. For every homomorphism
	$f : \nolinebreak F \rightarrow \nolinebreak K(\mathcal{F})$ there is a set
	$X \subseteq V(F)$ such that $f|_{X}$ is an isomorphism onto $K(\F)$.
	\end{proposition}
	\begin{proof}
	Let $C = C(F)$. Since $f|_{V(C)}$ is a homomorphism from $C$ to $K$, and since $K(\F)$ is maximal, we have $C \cong K(\F)$. By the property of a core, every homomorphism from $K(\F)$ to itself is an isomorphism. Since
	$C \cong K(\F)$, every homomorphism from $C$ to $K(\F)$ is an isomorphism. Thus, $f|_{V(C)}$ is an isomorphism and the assertion of the proposition holds with $X = V(C)$.
	\end{proof}
	
	\subsection{Proof of Theorems \ref{thm:finite_family} and  \ref{thm:one_graph_hard}}\label{subsec:hard_family}
	\noindent
	Theorems \ref{thm:finite_family} and \ref{thm:one_graph_hard} follow easily from the following theorem.
	\begin{theorem}\label{thm:aux_hard}
	For every $h \geq 3$ there is $\varepsilon_0 = \varepsilon_0(h)$ and $c = c(h)$ such that the following holds for every $\varepsilon < \varepsilon_0$ and for every non-bipartite graph $H$ on $h$ vertices. Let $K$ be the core of $H$. For every $n \geq n_0(\varepsilon)$ there is a graph on $n$ vertices with the following properties.
	\begin{enumerate}
	\item $G$ is homomorphic to $K$.
	\item $G$ is $\varepsilon$-far from being induced-$H$-free.
	\item $G$ contains at most $\varepsilon^{c\log(1/\varepsilon)}n^k$ (not necessarily induced) copies of $K$, where $k = |V(K)|$.
	\end{enumerate}
	\end{theorem}
	
	\begin{proof} % [Proof of Theorem \ref{thm:aux_hard}]
	Fix a homomorphism $\varphi : H \rightarrow K$. By assumption, $H$ is not bipartite. It is easy to see that the homomorphic image of a non-bipartite graph is itself non-bipartite, implying that $K$ contains an odd cycle. Label the vertices of $K$ by $a_1,\dots,a_k$ so that $a_1a_2 \dots a_ta_1$ is an odd cycle. Define
	$H_i = \varphi^{-1}(a_i)$ for $i = 1,\dots,k$. Label the vertices of $H$ by $1,\dots,h$ so that for each $1 \leq i < j \leq k$, the labels of the vertices in $H_i$ are smaller than the labels of the vertices in $H_j$.
	
	Let $\varepsilon > 0$. We will assume that $\varepsilon$ is small enough where needed. Let
	$R = R(h,h^2\varepsilon)$ be the graph obtained by applying Lemma \ref{lem:construction}. Here we assume that $\varepsilon < \frac{\delta_0(h)}{h^2}$, where $\delta_0(h)$ is from Lemma \ref{lem:construction}.
	Recall that $V(R) = V_1 \cup \dots \cup V_h$ and put $r = |V(R)|$. We define a graph $S$ on $V(R)$ as follows. By item 3 of Lemma \ref{lem:construction}, $R$ contains a collection $\mathcal{H}$ of at least $\varepsilon h^2 r^2$ pair-disjoint $h$-cliques, each of the form $\{v_1,\dots,v_h\}$ where
	$v_i \in V_i$. For every
	$\{v_1,\dots,v_h\} \in \mathcal{H}$ we put an induced copy of $H$ on $\{v_1,\dots,v_h\}$ in which $v_i$ plays the role of $i$ for every $i \in [h] = V(H)$. The resulting graph is $S$.
It is clear from the definition that ${\cal H}$ is a collection of pair-disjoint induced copies of $H$ in $S$.

%	Fix $\varepsilon_0 = \varepsilon_0(h)$ so that every
%	$\varepsilon < \varepsilon_0$ will satisfy the inequalities
%	\begin{equation}\label{eq:eps_0}
%	\log\left( \frac{1}{h\varepsilon} \right) > \frac{\log(1/\varepsilon)}{h}, \; \;
%	\varepsilon^{-1/2} > h\binom{h}{t},
%	\; \;
%	\varepsilon < \frac{\delta_0(h)}{h^2}
%	\end{equation}
%\begin{equation}\label{eq:eps_0}
%\varepsilon < \frac{\delta_0(h)}{h^2}, \; \; \; \; \;
%\binom{h}{t}\left(h^2\varepsilon \right)^
%{\beta\log\left( \frac{1}{h^2\varepsilon} \right)} \leq
%\varepsilon^{\frac{\beta}{2}\log(1/\varepsilon)},
%\end{equation}
%	where $\delta_0(h)$ and $\beta = \beta(h)$ are from Lemma \ref{lem:construction}.
%We prove the theorem with
%$$c = c(h) = \frac{\beta}{2}$$,
%where $\beta = \beta(h)$ is from Lemma \ref{lem:construction}.

	Let $n$ be an integer.
%	We will assume that $n$ is large enough where needed.
	For simplicity of presentation, we assume that $n$ is divisible by
	$r = |V(S)|$. Let $G$ be an $\frac{n}{r}$-blowup of $S$, that is, $G$ is obtained by replacing each vertex $v \in V(S)$ with an independent set $B(v)$ of size $\frac{n}{r}$, replacing edges with complete bipartite graphs and replacing non-edges with empty bipartite graphs. Clearly $|V(G)| = n$.
	For $1 \leq i \leq h$ put $B(V_i) = \bigcup_{v \in V_i}{B(v)}$.
	Observe that the map which sends $\bigcup_{i \in H_j}{B(V_i)}$ to $a_j$ for every $1 \leq j \leq k$ is a homomorphism from $G$ to $K$. This establishes item 3 in the statement of the theorem.
%	Indeed, the map that sends $B(V_i)$ to $i$ (for every $i = 1,\dots,h$) is a homomorphism from $G$ to $H$. Since there is a homomorphism from $H$ to $K$ (by the definition of a core), there is a homomorphism from $G$ to $K$.
	
	As mentioned above, ${\cal H}$ is a collection of at least $\varepsilon h^2 r^2$ pair-disjoint induced copies of $H$ in $S$. We call these copies the {\em base copies} of $H$.
%	By the definition of $S$, every $\{v_1,\dots,v_h\} \in {\cal H}$ spans an induced copy of $H$ in $S$. We call these copies the {\em base copies}.
	For every base copy $\{v_1,\dots,v_h\} \in {\cal H}$, Claim \ref{claim:disj_collection} gives a collection of at least $(n/rh)^2$ pair-disjoint induced copies of $H$ in $G$, each of the form $\{x_1,\dots,x_h\}$ with $x_i \in B(v_i)$. We say that these copies are {\em derived} from $\{v_1,\dots,v_h\}$. Since the base copies are pair-disjoint, two copies which are derived from different base copies are also pair-disjoint. Thus, $G$ contains a collection of at least
	$ |\mathcal{H}| \cdot (n/rh)^2 \geq
		\varepsilon h^2r^2 \cdot (n/rh)^2 = \varepsilon n^2$
pair-disjoint induced copies of $H$. This shows that $G$ is $\varepsilon$-far from being induced $H$-free.
	
To finish the proof it remains to show that $G$ contains at most $\varepsilon^{c\log(1/\varepsilon)}n^k$ copies of $K$. Consider a copy of $K$ in $G$. For each $j = 1,\dots,k$, let $U_j \subseteq V(G)$ be the set of vertices of this copy that are contained in
$\bigcup_{i \in H_j}{B(V_i)}$. Notice that the map that sends $U_j$ to $a_j$ (for each $j = 1,\dots,k$) is a homomorphism from $K$ to itself. By the property of a core (see Subsection \ref{subsec:hom}), this map is an isomorphism. Thus, $U_j = \{u_j\}$ and for each $1 \leq i < j \leq k$ we have $(u_i,u_j) \in E(G)$ if and only if $(a_i,a_j) \in E(K)$. Since $a_1a_2 \dots a_ta_1$ is a cycle in $K$, $u_1,\dots,u_t,u_1$ is a cycle in $G$. Let
$i_j \in H_j$ be such that $u_j \in B(V_{i_j})$ and let
$v_{i_j} \in V_{i_j}$ be such that $u_j \in B(v_{i_j})$. Then
$i_1 < i_2 < \dots < i_t$ due to the way we labeled the vertices of $H$. Since $G$ is a blowup of $S$, $v_{i_1}v_{i_2}\dots v_{i_t}v_{i_1}$ must be a cycle in $S$. Finally, by the definition of $S$,
$v_{i_1}v_{i_2}\dots v_{i_t}v_{i_1}$ is a cycle in $R$.
	
In the previous paragraph we proved that every copy of $K$ in $G$ contains vertices $u_1,\dots,u_t$ with the following property: there is an increasing sequence
$1 \leq i_1 < i_2 < \dots < i_t \leq h$ and vertices $v_{i_j} \in V_{i_j}$ such that $u_j \in B(v_{i_j})$ and $v_{i_1}v_{i_2}\dots v_{i_t}v_{i_1}$ is a cycle in $R$. For every increasing sequence $(i_1, i_2, \dots, i_t)$, Lemma \ref{lem:construction} states that $R$ contains at most $r^2$ cycles of the form $v_{i_1}v_{i_2}\dots v_{i_t}v_{i_1}$ with $v_{i_j} \in V_{i_j}$. Therefore, the number of copies of $K$ in $G$ that correspond to a specific increasing sequence is at most $r^2 \left( n/r \right)^t n^{k-t} \leq n^k/r$.
Here we used the inequality $t \geq 3$ which follows from the fact that the cycle $a_1,\dots,a_t$ is odd. We now take the union bound over all $\binom{h}{t}$ increasing sequences
$(i_1, i_2, \dots, i_t)$ and use the inequality
$r \geq (1/h^2\varepsilon)^{\beta(h)\log(1/h^2\varepsilon)}$ given by Lemma \ref{lem:construction}. We get that the number of copies of $K$ in $G$ is at most
$$\binom{h}{t}n^k/r \leq
\binom{h}{t}\left( h^2\varepsilon \right)^
{\beta(h) \log\left( 1/h^2\varepsilon \right)}n^k \leq \varepsilon^{\frac{\beta}{2}\log(1/\varepsilon)}n^k,$$ with the last inequality holding if $\varepsilon$ is small enough.
This completes the proof of the theorem.
\end{proof}
	\begin{proof}[Proof of Theorem \ref{thm:one_graph_hard}]
	By Theorem \ref{thm:aux_hard}, for every sufficiently small $\varepsilon > 0$ and for every $n \geq n_0(\varepsilon)$ there is a graph $G$ on $n$ vertices which is $\varepsilon$-far from being induced $H$-free yet contains at most $\varepsilon^{c\log(1/\varepsilon)}n^k$ (not necessarily induced) copies of $K$, the core of $H$. As $K$ is a subgraph of $H$, $G$ contains at most
	$\varepsilon^{c\log(1/\varepsilon)}n^k \cdot n^{h-k} = \varepsilon^{c\log(1/\varepsilon)} n^h$ (not necessarily induced) copies of $H$.
	\end{proof}
	\begin{proof}[Proof of Theorem \ref{thm:finite_family}]
	Write $\F = \{F_1,\dots,F_{\ell}\}$. By symmetry, it is enough to prove that $F_i$ is bipartite for some $1 \leq i \leq \ell$. Assume, by contradiction, that $F_i$ is not bipartite for every $1 \leq i \leq \ell$.
	We will show that for every sufficiently small $\varepsilon > 0$ and for every $n \geq n_0(\varepsilon)$, there is a graph $G$ which is $\varepsilon$-far from being induced $\F$-free and yet contains at most $\varepsilon^{c\log(1/\varepsilon)}n^{v(F_i)}$ copies of $F_i$ for every $1 \leq i \leq \ell$, where $c = c(\F)$ depends only on $\F$. This will imply that ${\cal P}^*_{\cal F}$ is not easily testable, in contradiction to the assumption of the theorem.
	
	Let $K = K(\F)$ be the graph defined in Subsection \ref{subsec:hom}. Let us assume, wlog, that $K$ is the core of $F_1$. We claim that the graph $G$, obtained by applying Theorem \ref{thm:aux_hard} to $H = F_1$ (and $K$), satisfies our requirements. Clearly, $G$ is $\varepsilon$-far from being induced $\F$-free, as it is $\varepsilon$-far from being induced $F_1$-free.
	
	By Theorem \ref{thm:aux_hard}, there is a homomorphism
	$g : G \rightarrow K$. Let $1 \leq i \leq \ell$ and consider an embedding $f : F_i \rightarrow G$ of $F_i$ into $G$. Then $g \circ f$ is a homomorphism from $F_i$ to $K$. By Proposition \ref{prop:core}, there is a set $X \subseteq V(F_i)$ such that $(g \circ f)|_{V(X)}$ is an isomorphism onto $K$. This means that $f(F_i) \subseteq V(G)$ contains a copy of $K$. We conclude that every copy of $F_i$ in $G$ contains a copy of $K$. By Theorem \ref{thm:aux_hard}, $G$ contains at most $\varepsilon^{c\log(1/\varepsilon)}n^k$ copies of $K$. It follows that $G$ contains at most $\varepsilon^{c\log(1/\varepsilon)} n^{v(F_i)}$ copies of $F_i$, as required.
	\end{proof}
	\subsection{Proof of Theorem \ref{thm:hard_superpoly_blowup}}\label{subsec:superpoly_blowup}
	Let $K$ be a graph with vertex set $[k]$. We say that a graph $F$ is a
	{\em blowup} of $K$ if $F$ admits a vertex-partition
	$V(F) = X_1 \cup \dots \cup X_k$ such that $X_1,\dots,X_k$ are independent sets and for every $1 \leq i < j \leq k$, if
	$(i,j) \in E(K)$ then $(X_i,X_j)$ is a complete bipartite graph and if $(i,j) \notin E(K)$ then $(X_i,X_j)$ is an empty bipartite graph.
%	(note that this definition is just a special case of Definition \ref{def:f_blowup} in which $g \equiv 0$).
	We say that $F$ is the {\em $s$-blowup} of $K$ if $\left| X_1 \right| = \dots = \left| X_k \right| = s$.
	
	Throughout this subsection, $C_m$ denotes the cycle of length $m$. In the proof of Theorem \ref{thm:hard_superpoly_blowup} we use the following simple proposition, whose proof appears at the end of this subsection.
	\begin{proposition}\label{prop:odd_cycle_blowup}
		Let $k$ be an odd integer and let $G$ be a blowup of $C_k$. Then $G$ is induced $C_6$-free and (not necessarily induced) $C_{\ell}$-free for every odd $3 \leq \ell < k$.
	\end{proposition}
%	We denote by $K(s)$ the {\em $s$-blowup} of $K$, namely, the blowup $F$ in which
%	$\left| X_1 \right| = \dots = \left| X_k \right| = s$.
	
	Recall the definition of a graph homomorphism from Subsection \ref{subsec:hom}. We will use the simple fact that $C_{2\ell+1}$ has a homomorphism into $C_{2k+1}$ for every $\ell \geq k$.
	For the proof of Theorem \ref{thm:hard_superpoly_blowup} we need the following lemma from \cite{separation}.
	\begin{lemma}\label{lem:blowup_distance}\cite{separation}
		Let $K$ be a graph on $k$ vertices, let $F$ be a graph on $f$ vertices which has a homomorphism into $K$ and let $G$ be the $\frac{n}{k}$-blowup of $K$ where $n \geq n_0(f)$. Then $G$ is $\frac{1}{2k^2}$-far from being (not necessarily induced) $F$-free.
	\end{lemma}
	
	For a graph $F$, denote by $SG(F)$ the set of supergraphs of $F$ (namely, the set of all graphs on $V(F)$ obtained from $F$ by adding edges). Note that being (not necessarily induced) $F$-free is equivalent to being induced $SG(F)$-free. We are now ready to prove Theorem \ref{thm:hard_superpoly_blowup}.
	\begin{proof}[Proof of Theorem \ref{thm:hard_superpoly_blowup}]
		Define a sequence $\{a_i\}_{i \geq 1}$ as follows: set $a_1 = 3$ and
		$a_{i+1} = 2^{2(a_{i} + 2)^2} + 1$. Note that $a_i$ is odd for every $i \geq 1$.
%		In fact, we could have replaced $2^k$ with any function of $k$ that grows faster than any polynomial in $k$.
		We prove the theorem with the graph family
		$$ \F =
		\{C_6\} \cup
		\bigcup_{i \geq 1}{SG \big( C_{a_i} \big)}.$$
		Since $a_1 = 3$ we have $C_3 \in \F$.
		 Note that $C_6$ is a bipartite graph and that $C_3$ is both a co-bipartite graph and a split graph. For $i \geq 1$ put
		 $\varepsilon_i = \frac{1}{2(a_i + 2)^2}$. We will show that $f_{{\cal P}^*_{\cal F}}\left( \varepsilon_i \right) \geq 2^{1/\varepsilon_i}$ for every $i \geq 1$ (recall Definition \ref{def:test}), which implies that
		 ${\cal P}^*_{\cal F}$ is not easily testable.
		
		Let $i \geq 1$ and put $k = a_i + 2$ and $f = a_{i+1}$. Since $a_i$ is odd and $a_i \geq 3$, we have that $k$ is odd and $k \geq 5$. Let $n \geq n_0(f)$ which is divisible by $k$ (where $n_0(f)$ is from Lemma \ref{lem:blowup_distance}) and let $G$ be an $\frac{n}{k}$-blowup of $C_k$. By our choice of $\varepsilon_i$ and $k$ we have
		$\varepsilon_i = \frac{1}{2k^2}$. Since $C_f$ has a homomorphism into $C_k$, Lemma \ref{lem:blowup_distance} applies that $G$ is $\varepsilon_i$-far from being $C_f$-free and hence is $\varepsilon_i$-far from being induced $SG(C_f)$-free. As
		$SG(C_f) \subseteq \F$, we conclude that $G$ is $\varepsilon_i$-far from being induced $\F$-free.
		
		Proposition \ref{prop:odd_cycle_blowup} implies that $G$ is induced $C_6$-free and that for every odd $3 \leq \ell < k$, $G$ is $C_{\ell}$-free and hence induced $SG(C_{\ell})$-free. By the definition of $\F$, if $F \in \F$ is an induced subgraph of $G$ then
		$|V(F)| \geq a_{i+1} > 2^{2(a_{i} + 2)^2} = 2^{1/\varepsilon_i}.$ Here we used the definition of the sequence $\{a_i\}_{i \geq 1}$ and our choice of $\varepsilon_i$.
		We conclude that every set $Q \subseteq V(G)$ of size less than $2^{1/\varepsilon_i}$ is induced $\F$-free, implying that
		$f_{{\cal P}^*_{\cal F}}\left( \varepsilon_i \right) > 2^{1/\varepsilon_i}$, as required.
	\end{proof}
	We remark that using essentialy the same proof as above, we could have proven the following strengthening of Theorem \ref{thm:hard_superpoly_blowup}. For every function
	$g : (0,1/2) \rightarrow \mathbb{N}$ there is a graph family $\F$ that contains a bipartite graph, a co-bipartite graph and a split graph, and there is a decreasing sequence $\{\varepsilon_i\}_{i \geq 1}$ with $\varepsilon_i \rightarrow 0$ such that
	$f_{{\cal P}^*_{\cal F}}(\varepsilon_i) > g(\varepsilon_i)$ for every
	$i \geq 1$.
	
	\begin{proof}[Proof of Proposition \ref{prop:odd_cycle_blowup}]
	As $G$ is a blow-up of $C_k$, it has a partition
	$V(G) = X_1 \cup \dots \cup X_k$ into independent sets such that $(X_i,X_j)$ is a complete bipartite graph if
	$|i-j| \equiv \pm 1 \pmod{k}$ and an empty bipartite graph otherwise. Assume, by contradiction, that
	there is $Z \subseteq V(G)$ such that $G[Z] \cong C_6$. Since $C_6$ is not a subgraph of $C_k$, there must be $1 \leq i \leq k$ such that
	$|Z \cap X_i| \geq 2$. Assume wlog that there are distinct
	$u,v \in Z \cap X_1$. By the structure of $C_6$, there are distinct $x,y \in Z$ such that $(u,x), (u,y) \in E(G)$. Then
	$x,y \in X_2 \cup X_k$, implying that $(v,x), (v,y) \in E(G)$. Thus, $uxvy$ is a $4$-cycle, in contradiction to the fact that
	$G[Z] \cong C_6$.
	
	For the second part of the proposition, simply observe that every subgraph of $G$ with less than $k$ vertices is bipartite.
%	Now let $3 \leq \ell < k$ be an odd integer and let us show that $G$ is $C_{\ell}$-free. Let $Z \subseteq V(G)$ of size $|Z| = \ell$. Since
%	$\ell < k$, there must be
%	$1 \leq i \leq k$ such that $X_i \cap \{v_1,\dots,v_{\ell}\} = \emptyset$, say wlog that $i = 1$. As $G[X_2 \cup \dots \cup X_k]$ is a blowup of a path, it is a bipartite graph. Since $Z \subseteq X_2 \cup \dots \cup X_k$, the graph $G[Z]$ is bipartite. This shows that $G$ is $C_{\ell}$-free as $\ell$ is odd.
	\end{proof}
	
\end{document}